\documentclass[a4paper,11pt,reqno]{amsart}
\usepackage{amsmath}%
\usepackage{amsthm}%
\usepackage{amsfonts}%
\usepackage{amssymb}%
\usepackage[english]{babel}
\usepackage{graphicx}%
\usepackage{easymat}%
\usepackage{bbm}
\usepackage{inputenc}
\usepackage{a4wide}
\usepackage{cite}
\usepackage{epsfig}
\usepackage{color}

\usepackage{etex}

\usepackage{todonotes}
\usepackage{hyperref}
\usepackage[all]{xy}
\newtheorem{theorem}{Theorem}[section]

\newtheorem{assumption}[theorem]{Assumption}

\newtheorem{corollary}[theorem]{Corollary}

\newtheorem{definition}[theorem]{Definition}

\newtheorem{lemma}[theorem]{Lemma}

\newtheorem{proposition}[theorem]{Proposition}
\newtheorem{remark}[theorem]{Remark}

\numberwithin{equation}{section}

\def\beq{\begin{equation}}
\def\ee{\end{equation}}

\def\CC{\mathbb{C}}
\def\DD{\mathbb{D}}
\def\RR{\mathbb{R}}
\def\NN{\mathbb{N}}
\def\ZZ{\mathbb{Z}}
\def\SS{\mathbb{S}}

\def\HH{\mathbb{H}}

\def\PP{\mathbb{P}}

\def\unity{\mathbbm{1}}
\def\del{\partial}
\def\delbar{\bar{\partial}}

\def\trace{\text{tr}}

\begin{document}
\title{Periodic solutions of the sinh-Gordon equation and integrable systems}

\date{\today}

\author{Markus Knopf}
%\author{Martin U. Schmidt}

\address{Institut f\"ur Mathematik, Universit\"at Mannheim, 
68131 Mannheim, Germany.} 

\email{knopf@math.uni-mannheim.de}
%\email{schmidt@math.uni-mannheim.de}

\begin{abstract}
\footnotesize
We study the space of periodic solutions of the elliptic $\sinh$-Gordon equation by means of spectral data consisting of a Riemann surface $Y$ and a divisor $D$. We show that the space $M_g^{\mathbf{p}}$ of real periodic finite type solutions with fixed period $\mathbf{p}$ can be considered as a completely integrable system $(M_g^{\mathbf{p}},\Omega,H_2)$ with a symplectic form $\Omega$ and a series of commuting Hamiltonians $(H_n)_{n \in \NN}$. In particular we relate the gradients of these Hamiltonians to the Jacobi fields $(\omega_n)_{n\in \NN_0}$ from the Pinkall-Sterling iteration. Moreover, a connection between the symplectic form $\Omega$ and Serre duality is established.
\end{abstract}

\maketitle

\setcounter{tocdepth}{1}
\tableofcontents

\section{Introduction}

The elliptic \textit{sinh-Gordon equation} is given by
\beq
\Delta u + 2\sinh(2u) = 0,
\label{eq_sinh}
\ee
where $\Delta$ is the Laplacian of $\RR^2$ with respect to the Euclidean metric and $u:\RR^2 \to \RR$ is a twice partially differentiable function which we assume to be real. \\

The $\sinh$-Gordon equation arises in the context of particular surfaces of constant mean curvature (CMC) since the function $u$ can be extracted from the conformal factor $e^{2u}$ of a conformally parameterized CMC surface. The study of CMC tori in $3$-dimensional space forms was strongly influenced by algebro-geometric methods (as described in \cite{Bobenko_Its_Matveev}) that led to a complete classification by Pinkall and Sterling \cite{Pinkall_Sterling} for CMC-tori in $\RR^3$. Moreover, Bobenko \cite{bob3, bob2} gave explicit formulas for CMC tori in $\RR^3$, $\SS^3$ and $\HH^3$ in terms of theta-functions and introduced a description of such tori by means of \textit{spectral data}. We also refer the interested reader to \cite{bob1,bob4}. Every CMC torus yields a doubly periodic solution $u:\RR^2 \to \RR$ of the $\sinh$-Gordon equation.
 With the help of differential geometric considerations one can associate to every CMC torus a hyperelliptic Riemann surface $Y$, the so-called \textit{spectral curve}, and a holomorphic line bundle $E$ on $Y$ (the so-called \textit{eigenline bundle}) that is represented by a certain divisor $D$. Hitchin \cite{Hitchin_Harmonic}, and Pinkall and Sterling \cite{Pinkall_Sterling} independently proved that all doubly periodic solutions of the $\sinh$-Gordon equation correspond to spectral curves of finite genus. We say that solutions of \eqref{eq_sinh} that correspond to spectral curves of finite genus are of \textit{finite type}. \\

In the present setting we will relax the condition on the periodicity and demand that $u$ is only simply periodic with a fixed period. After rotating the domain of definition we can assume that this period is real. This enables us to introduce simply periodic \textit{Cauchy data} with fixed period $\mathbf{p} \in \RR$ consisting of a pair $(u,u_y) \in C^{\infty}(\RR/\mathbf{p}\ZZ) \times C^{\infty}(\RR/\mathbf{p}\ZZ)$. Moreover, we demand that the corresponding solution $u$ of the $\sinh$-Gordon equation is of finite type. \\

This paper summarizes the author's PhD thesis \cite{Knopf_phd}. Its main goal is to identify the $\sinh$-Gordon equation \eqref{eq_sinh} as a completely integrable system (compare with \cite{Pedit_Schmitt}) and illustrate its features in the finite type situation. 
%Moreover, we introduce the \textit{spectral data} $(Y,D)$ and investigate how $Y$ and $D$ fit into the description of the $\sinh$-Gordon equation as a completely integrable system.

\section{Conformal CMC immersions into $\SS^3$}

\subsection{The Lie groups $SL(2,\CC)$ and $SU(2)$} 
Let us consider the Lie group $SL(2,\CC) := \{A \in M_{2\times 2}(\CC) \left| \right. \det(A) = 1 \}$. The Lie algebra $\mathfrak{sl}_2(\CC) :=  \{B \in M_{2\times 2}(\CC) \left| \right. \trace(B) = 0\}$ of $SL(2,\CC)$ is spanned by the matrices $\epsilon_+,\epsilon_-, \epsilon$ with
$$
\epsilon_+ = \begin{pmatrix} 0 & 1 \\ 0 & 0 \end{pmatrix},\;\; \epsilon_- = \begin{pmatrix} 0 & 0 \\ 1 & 0 \end{pmatrix},\;\; \epsilon = \begin{pmatrix} i & 0 \\ 0 & -i \end{pmatrix}.
$$

It will also be convenient to identify $\SS^3$ with the Lie group $SU(2) = \{A \in M_{2\times 2}(\CC) \left| \right. \det(A) = 1, \; \bar{A}^t = A^{-1} \}$. The Lie algebra of $SU(2)$ is denoted by $\mathfrak{su}(2)$ and a direct computation shows that $\mathfrak{su}(2) = \{B \in M_{2\times 2}(\CC) \left| \right. \trace(B) = 0, \; \bar{B}^t = -B \} \simeq \RR^3$.

\subsection{Extended frames}
We start with the following version of a result by Bobenko \cite{bob2} (cf. \cite{Kilian_Schmidt_infinitesimal}, Theorem 1.1).

\begin{theorem}\label{Alpha_mit_Lambda}
 Let $u: \CC \to  \RR$ and $Q: \CC \to \CC$ be smooth functions and define
\[
 \alpha_{\lambda} = \frac{1}{2}
\begin{pmatrix}
 u_z dz - u_{\bar{z}} d\bar{z} & i\lambda^{-1}e^{u}dz + i\overline{Q}e^{-u} d\bar{z} \\
 iQe^{-u} dz + i\lambda e^u d\bar{z} & -u_z dz + u_{\bar{z}}d\bar{z}
\end{pmatrix}.
\]
Then $2d\alpha_{\lambda} + [\alpha_{\lambda} \wedge \alpha_{\lambda}] = 0$ if and only if $Q$ is holomorphic, i.e. $Q_{\bar{z}} = 0$, and $u$ is a solution of the reduced Gauss equation
\beq
2u_{z\bar{z}} + \tfrac{1}{2}(e^{2u} - Q\overline{Q}e^{-2u}) = 0. \nonumber
\ee
For any solution $u$ of the above equation and corresponding extended frame $F_\lambda$, and $\lambda_0,\lambda_1 \in \SS^1, \lambda_0 \neq \lambda_1$, i.e. $\lambda_k=e^{it_k}$ the map defined by the \textit{Sym-Bobenko-formula}
\beq
f = F_{\lambda_1}F_{\lambda_0}^{-1} \nonumber
\ee
is a conformal immersion $f: \CC \to SU(2) \simeq \SS^3$ with constant mean curvature
\beq
H = i\frac{\lambda_0 + \lambda_1}{\lambda_0 - \lambda_1} = \cot(t_0 - t_1), \nonumber
\ee
conformal factor $v = e^u / \sqrt{H^2+1}$, and Hopf differential $\widetilde{Q}dz^2$ with $\widetilde{Q} = -\frac{i}{4}(\lambda_1^{-1} - \lambda_0^{-1})Q$.
\end{theorem}

\begin{assumption} \label{assumption_Q_const}
Let us assume that the Hopf differential $Q$ is constant with $|Q| = 1$.
\end{assumption}

\subsection{The monodromy}

The central object for the following considerations is the monodromy $M_{\lambda}$ of a frame $F_{\lambda}$.

\begin{definition}
Let $F_{\lambda}$ be an extended frame and assume that $\alpha_{\lambda} = F_{\lambda}^{-1}dF_{\lambda}$ has period $\mathbf{p} \in \CC$, i.e. $\alpha_{\lambda}(z+\mathbf{p}) = \alpha_{\lambda}(z)$. Then the \textbf{monodromy} of the frame $F_{\lambda}$ with respect to the period $\mathbf{p}$ is given by
$$
M_{\lambda}^{\mathbf{p}} := F_{\lambda}(z+\mathbf{p})F_{\lambda}^{-1}(z).
$$
\end{definition}

Note that we have $d M_{\lambda}^{\mathbf{p}} = 0$
%\begin{eqnarray*}
%d M_{\lambda}^{\mathbf{p}} & = & F_{\lambda}(z+\mathbf{p}) \alpha_{\lambda}(z+\mathbf{p})F_{\lambda}^{-1}(z) - F_{\lambda}(z+\mathbf{p}) \alpha_{\lambda}(z)F_{\lambda}^{-1}(z) \\
%& = & 0,
%\end{eqnarray*}
%since $\alpha_{\lambda}(z+\mathbf{p}) = \alpha_{\lambda}(z)$ 
and thus $M_{\lambda}^{\mathbf{p}}$ does not depend on $z$. Setting $F_{\lambda}(0) = \unity$ we get 
$$
M_{\lambda} := M_{\lambda}^{\mathbf{p}} = F_{\lambda}(\mathbf{p})F_{\lambda}^{-1}(0) = F_{\lambda}(\mathbf{p}).
$$

\subsection{Isometric normalization}
We can rotate the coordinate $z$ by a map $z \mapsto w(z) = e^{i\varphi}z$ in such a way that $\Im(\mathbf{p}) = 0$. As a consequence the extended frame $F_{\lambda}$ is multiplied with the matrix
%Since
%$$
%u(z) = \widetilde{u}(e^{i\varphi}z) + \ln(|\tfrac{d}{dz}(e^{i\varphi}z)|) = \widetilde{u}(e^{i\varphi}z)
%$$
%we get
%$$
%\widetilde{u}(e^{i\varphi}z) = u(z) = u(z+\widetilde{\mathbf{p}}) = \widetilde{u}(e^{i\varphi}(z+\widetilde{\mathbf{p}})) = \widetilde{u}(e^{i\varphi}z+\mathbf{p})
%$$
%for a suitable $\varphi \in [0,2\pi)$ and $\mathbf{p} = e^{i\varphi}\widetilde{\mathbf{p}}$. 
$$
B_{w} = \begin{pmatrix} \delta^{-1/2} & 0 \\ 0 & \delta^{1/2} \end{pmatrix}
$$
with $\delta = e^{i\varphi} \in \SS^1$. This corresponds to the isometric normalization described in \cite{Hauswirth_Kilian_Schmidt_1}, Remark 1.5, and the corresponding gauged $\alpha_{\lambda}$ is of the form
\beq
 \alpha_{\lambda} = \frac{1}{2}
\begin{pmatrix}
 u_z dz - u_{\bar{z}} d\bar{z} & i\lambda^{-1}\delta e^{u}dz + i\overline{\gamma}e^{-u} d\bar{z} \\
 i\gamma e^{-u} dz + i\lambda \bar{\delta} e^u d\bar{z} & -u_z dz + u_{\bar{z}}d\bar{z}
\end{pmatrix},
\label{eq_alpha}
\ee
where the constant $\gamma \in \SS^1$ is given by $\gamma = \delta^{-1} Q = \bar{\delta}Q$.

\subsection{The sinh-Gordon equation} 
We can normalize the above parametrization with $\delta = 1$ and $|\gamma| = 1$ in \eqref{eq_alpha} by choosing the appropriate value for $Q \in \SS^1$. Then we can consider the system
$$
dF_{\lambda} = F_{\lambda}\alpha_{\lambda} \;\; \text{ with } \;\;  F_{\lambda}(0) = \unity.
$$
%for
%$$
%F(z,\lambda): \CC \times \CC^* \to SL(2,\CC)
%$$
%and
%\[
% \alpha_{\lambda} = \frac{1}{2}
%\begin{pmatrix}
% u_z dz - u_{\bar{z}} d\bar{z} & i\lambda^{-1} e^{u}dz + i\overline{\gamma}e^{-u} d\bar{z} \\
% i\gamma e^{-u} dz + i\lambda e^u d\bar{z} & -u_z dz + u_{\bar{z}}d\bar{z}
%\end{pmatrix}.
%\]
Since $|\gamma| = 1$, wee see that the compatibility condition $2d\alpha_{\lambda} + [\alpha_{\lambda} \wedge \alpha_{\lambda}] = 0$ from Theorem \ref{Alpha_mit_Lambda} holds if and only if 
\beq
2u_{z\bar{z}} + \tfrac{1}{2}(e^{2u} - \gamma\overline{\gamma}e^{-2u}) = 2u_{z\bar{z}} + \sinh(2u) = 0.
\label{eq_sinh}
\ee
Thus the reduced Gauss equation turns into the \textbf{sinh-Gordon equation} in that situation. 
%The monodromy of $F$ is then $M_{\lambda} = F(\mathbf{p},\lambda)$ for a period $\mathbf{p}$ of the solution $u$ of the $\sinh$-Gordon equation. 
For the following we make an additional assumption.

\begin{assumption}\label{assumption_gamma_1}
Let $\gamma = \delta = 1$. This yields
\[
 \alpha_{\lambda} = \frac{1}{2}
\begin{pmatrix}
 u_z dz - u_{\bar{z}} d\bar{z} & i\lambda^{-1}e^u dz + ie^{-u}d\bar{z} \\
ie^{-u}dz + i\lambda e^u d\bar{z} & -u_z dz + u_{\bar{z}}d\bar{z}
\end{pmatrix}.
\]
\end{assumption}

If we evaluate $\alpha_{\lambda}$ along the vector fields $\tfrac{\partial}{\partial x}$ and $\tfrac{\partial}{\partial y}$ we obtain
$$
U_{\lambda} := \alpha_{\lambda}(\tfrac{\partial}{\partial x}), \;\;\;\; V_{\lambda} := \alpha_{\lambda}(\tfrac{\partial}{\partial y}).
$$
These matrices will be important for the upcoming considerations. In particular $U_{\lambda}$ reads
\beq
U_{\lambda} = \frac{1}{2}
\begin{pmatrix}
 -i u_y & i\lambda^{-1}e^u + ie^{-u} \\
i\lambda e^u + ie^{-u} & iu_y
\end{pmatrix}.
\label{eq_U_lambda}
\ee

\begin{remark}\label{isomorphism_U_lambda}
%Due to the one-to-one correspondence $(u,u_y) \mapsto U_{\lambda}$ with 
%$$
%U_{\lambda}=
%\frac{1}{2}
%\begin{pmatrix}
% -i u_y & i\lambda^{-1}e^u + ie^{-u} \\
%i\lambda e^u + ie^{-u} & iu_y
%\end{pmatrix}
%$$
Due to \eqref{eq_U_lambda} we can identify the tuple $(u,u_y)$ with the matrix $U_{\lambda}$.
\end{remark}

\section{A formal diagonalization of the monodromy $M_{\lambda}$}
We want to diagonalize the monodromy $M_{\lambda}$ and therefore need to diagonalize $\alpha_{\lambda}$. A diagonalization for the Schr\"odinger-operator is done in \cite{Schmidt_infinite} based on a result from \cite{Haak_Schmidt_Schrader}. In order to adapt the techniques applied there we search for a $\lambda$-dependent periodic formal power series $\widehat{g}_{\lambda}(x)$ such that
$$
\widehat{\beta}_{\lambda} = \widehat{g}_{\lambda}^{-1}\alpha_{\lambda}\widehat{g}_{\lambda} + \widehat{g}_{\lambda}^{-1}\tfrac{d}{dx}\widehat{g}_{\lambda}
$$
is a diagonal matrix, i.e. 
$$
\widehat{\beta}_{\lambda}(x) = \begin{pmatrix} \sum_m (\sqrt{\lambda})^m b_m(x) & 0 \\ 0 & -\sum_m (\sqrt{\lambda})^m b_m(x) \end{pmatrix}
$$
with $m \geq -1$. Since $F_{\lambda}(x) = \widehat{g}_{\lambda}(0)\widehat{G}_{\lambda}(x)\widehat{g}_{\lambda}(x)^{-1}$ (where $\widehat{G}_{\lambda}$ solves $\frac{d}{dx} \widehat{G}_{\lambda}(x) = \widehat{G}_{\lambda}(x) \widehat{\beta}_{\lambda}(x)$ with $\widehat{G}_{\lambda}(0) = \unity$) we get 
$$
M_{\lambda} = F_{\lambda}(\mathbf{p}) = \widehat{g}_{\lambda}(0)\widehat{G}_{\lambda}(\mathbf{p})\widehat{g}_{\lambda}(\mathbf{p})^{-1} = \widehat{g}_{\lambda}(0)\widehat{G}_{\lambda}(\mathbf{p})\widehat{g}_{\lambda}(0)^{-1}
$$
and due to
\[
\widehat{G}_{\lambda}(x) = \begin{pmatrix} \exp\left(\int_0^x \sum_m (\sqrt{\lambda})^m b_m(t)\, dt\right) & 0 \\ 0 & \exp\left(-\int_0^x \sum_m (\sqrt{\lambda})^m b_m(t)\, dt\right) \end{pmatrix}
\]
%one has
%\[
%\widehat{G}_{\lambda}(\mathbf{p}) = \begin{pmatrix} \exp\left(\int_0^{\mathbf{p}} \sum_m (\sqrt{\lambda})^m b_m(t)\, dt\right) & 0 \\ 0 & \exp\left(-\int_0^{\mathbf{p}} \sum_m (\sqrt{\lambda})^m b_m(t)\, dt\right) \end{pmatrix}.
%\]
%The conjugation with the matrix $\widehat{g}_{\lambda}(0)$ leaves the eigenvalues $\mu, \frac{1}{\mu}$ of $M_{\lambda}$ invariant and thus 
we obtain for the eigenvalue $\mu$ of $M_{\lambda}$
$$
\mu = \exp\left(\int_0^{\mathbf{p}} \sum_m (\sqrt{\lambda})^m b_m(t)\, dt\right) \;\;\; \text{ or equivalently } \;\;\; \ln\mu = \sum_m (\sqrt{\lambda})^m \int_0^{\mathbf{p}} b_m(t)\, dt.
$$
Let us start with the following lemma that is obtained by a direct calculation.

\begin{lemma}\label{lemma_first_gauge}
By performing a gauge transformation with
\[
g_{\lambda}(z) = \frac{1}{\sqrt{2}}\begin{pmatrix} e^{\frac{u}{2}} & 0 \\ 0 & \sqrt{\lambda}e^{-\frac{u}{2}}\end{pmatrix}\begin{pmatrix} 1 & -1 \\ 1 & 1 \end{pmatrix}
\]
the frame $F_{\lambda}(z)$ is transformed into $F_{\lambda}(z)g_{\lambda}(z)$ and the map $G_{\lambda}(z) := g_{\lambda}(0)^{-1}F_{\lambda}(z)g_{\lambda}(z)$ solves
$$
dG_{\lambda} = G_{\lambda}\beta_{\lambda} \;\; \text{with } \beta_{\lambda} = g_{\lambda}^{-1}\alpha_{\lambda} g_{\lambda} + g_{\lambda}^{-1}dg_{\lambda} \text{ and } G_{\lambda}(0) = \unity.
$$
Evaluating the form $\beta_{\lambda}$ along the vector field $\frac{\partial}{\partial x}$ and setting $y=0$ 
yields $\beta_{\lambda}(\frac{\partial}{\partial x}) = \tfrac{1}{\sqrt{\lambda}}\beta_{-1} + \beta_0 + \sqrt{\lambda}\beta_1$ with
\[
 \beta_{-1} = \begin{pmatrix} \tfrac{i}{2} & 0 \\ 0 & -\tfrac{i}{2} \end{pmatrix}, \;\; \beta_0 = \begin{pmatrix} 0 & -u_z \\ -u_z & 0 \end{pmatrix}, \;\; \beta_1 = \begin{pmatrix} \frac{i}{2}\cosh(2u) & - \frac{i}{2}\sinh(2u) \\ \frac{i}{2}\sinh(2u) &  -\frac{i}{2}\cosh(2u) \end{pmatrix}.
\]
\end{lemma}

From the following theorem we obtain a periodic formal power series $\widetilde{g}_{\lambda}(x)$ of the form $\widetilde{g}_{\lambda}(x) = \unity + \sum_{m\geq 1} a_m(x)(\sqrt{\lambda})^m$ such that $\widehat{g}_{\lambda}(x) := g_{\lambda}(x) \widetilde{g}_{\lambda}(x)$ (with $g_{\lambda}(x)$ defined as in Lemma \ref{lemma_first_gauge}) diagonalizes $\alpha_{\lambda}$ around $\lambda = 0$.

\begin{theorem}\label{expansion}
Let $(u,u_y) \in C^{\infty}(\RR/\mathbf{p} \ZZ) \times C^{\infty}(\RR/\mathbf{p} \ZZ)$. Then there exist two series
\begin{gather*}
 a_1(x), a_2(x),\ldots \;\;\; \in \text{span}\{\epsilon_+,\epsilon_-\} \text{ of periodic off-diagonal matrices and} \\
 b_1(x), b_2(x),\ldots \;\;\; \in \text{span}\{\epsilon\} \text{ of periodic diagonal matrices, respectively}
\end{gather*}
such that $a_{m+1}(x)$ and $b_m(x)$ are differential polynomials in $u$ and $u_y$ with derivatives of order $m$ at most and the following equality for formal power series holds asympotically around $\lambda = 0$:
\begin{gather}
\beta_{\lambda}(x)\left(\unity + \sum_{m\geq 1} a_m(x)(\sqrt{\lambda})^m\right) + \sum_{m\geq 1} \frac{d}{dx}a_m(x)(\sqrt{\lambda})^m = \hspace{4cm} \nonumber \\ 
\hspace{5cm} \left(\unity + \sum_{m\geq 1} a_m(x)(\sqrt{\lambda})^m\right) \sum_{m \geq -1} b_m(x) (\sqrt{\lambda})^m.
\tag{$*$}
\label{eq_expansion}
\end{gather}
Here $b_{-1}(x)$ and $b_0(x)$ are given by $b_{-1}(x) \equiv \beta_{-1} = \tfrac{i}{2}\left(\begin{smallmatrix} 1 & 0 \\ 0 & -1 \end{smallmatrix}\right)$ and $b_0(x) \equiv \left(\begin{smallmatrix} 0 & 0 \\ 0 & 0 \end{smallmatrix}\right)$.
\end{theorem}

\begin{remark} 
Since we only consider finite type solutions we can guarantee that the power series in \eqref{eq_expansion} indeed are convergent, see \cite[Theorem 4.35]{Knopf_phd}.
\end{remark}

\begin{proof}
 We start the iteration with $b_0(x) \equiv \left(\begin{smallmatrix} 0 & 0 \\ 0 & 0 \end{smallmatrix}\right)$ and will inductively solve the given ansatz in all powers of $\sqrt{\lambda}$:
\begin{enumerate}
\item $(\sqrt{\lambda})^{-1}$: $\beta_{-1} = \beta_{-1}$. $\checkmark$
\item $(\sqrt{\lambda})^0$: $\beta_{-1}a_1(x) + \beta_0(x) = b_0(x) = 0$ and thus 
$$
a_1(x) = -\beta_{-1}^{-1}\beta_0(x) = \begin{pmatrix} 0 & -i\del u \\ i\del u & 0 \end{pmatrix}.
$$
\item $(\sqrt{\lambda})^{1}$: $\beta_{-1}a_2(x) + \beta_0(x)a_1(x) + \beta_1(x) + \frac{d}{dx}a_1(x) = b_1(x) + a_2(x)\beta_{-1} + a_1(x)b_0(x)$. Rearranging terms and sorting with respect to diagonal ($\mathrm{d}$) and off-diagonal ($\mathrm{off}$) matrices we get two equations:
\begin{eqnarray*}
b_1(x) & = & \beta_0(x)a_1(x) + \beta_{1,\mathrm{d}}(x) \\
& = & \begin{pmatrix} 
-i(\del u)^2 + \frac{i}{2}\cosh(2u) & 0 \\ 
0 & i(\del u)^2 -\frac{i}{2} \cosh(2u) 
\end{pmatrix}, \cr
[\beta_{-1},a_2(x)] & = & -\beta_{1,\mathrm{off}}(x) - \tfrac{d}{dx}a_1(x).
\end{eqnarray*}
In order to solve the second equation for $a_2(x)$ we make the following observation: set $a(x) = a_+(x) \epsilon_+ + a_-(x)\epsilon_-$. Since $[\epsilon,\epsilon_+] = 2i\epsilon_+$ and $[\epsilon,\epsilon_-] = -2i\epsilon_-$, we can define a linear map $\phi: \text{span} \{\epsilon_+,\epsilon_-\} \to \text{span} \{\epsilon_+,\epsilon_-\}$ by
$$
\phi(a(x)) := [\beta_{-1},a(x)] = [\tfrac{1}{2}\epsilon, a(x)] = ia_+(x)\epsilon_+ - ia_-(x)\epsilon_- \in \text{span} \{\epsilon_+,\epsilon_-\}.
$$
Obviously $\ker(\phi) = \{ 0\}$ and thus $\phi$ is an isomorphism. Therefore we can uniquely solve the equation $[\beta_{-1},a_2(x)] = -\beta_{1,\mathrm{off}}(x) - \tfrac{d}{dx}a_1(x)$ and obtain $a_2(x)$.
\end{enumerate}
We now proceed inductively for $m\geq 2$ and assume that we already found $a_m(x)$ and $b_{m-1}(x)$. Consider the equation
\begin{gather*}
\beta_{-1}a_{m+1}(x) + \beta_0(x)a_m(x) + \beta_1(x)a_{m-1}(x) + \tfrac{d}{dx} a_m(x) = \hspace{2cm} \\
\hspace{5cm} b_m(x) + a_{m+1}(x)\beta_{-1} + \sum_{i=1}^m a_i(x)b_{m-i}(x)
\end{gather*}
for the power $(\sqrt{\lambda})^m$. Rearranging terms and after decomposition in the diagonal ($\mathrm{d}$) and off-diagonal ($\mathrm{off}$) part we get
\begin{eqnarray*}
b_m(x) & = & \beta_0(x)a_m(x) + \beta_{1,\mathrm{off}}(x)a_{m-1}(x), \cr
[\beta_{-1},a_{m+1}(x)] & = & - \beta_{1,\mathrm{d}}(x)a_{m-1}(x)- \tfrac{d}{dx}a_m(x)  + \sum_{i=1}^m a_i(x)b_{m-i}(x).
\end{eqnarray*}
From the discussion above we see that these equations can uniquely be solved and one obtains $a_{m+1}(x)$ and $b_m(x)$. By induction one therefore obtains a unique formal solution of \eqref{eq_expansion} with the desired properties.
\end{proof}

With the help of Theorem \ref{expansion} we can reproduce Proposition 3.6 presented in \cite{Kilian_Schmidt_infinitesimal}.

\begin{corollary}\label{expansion_mu}
The logarithm $\ln\mu$ of the eigenvalue $\mu$ of the monodromy $M_{\lambda}$ has the following asymptotic expansion
\begin{eqnarray*}
\ln\mu & = & \tfrac{1}{\sqrt{\lambda}}\tfrac{i\mathbf{p}}{2} + \sqrt{\lambda}\int_0^{\mathbf{p}}\left(-i(\del u)^2 + \tfrac{i}{2}\cosh(2u)\right)\,dt + O(\lambda) \; \text{ at } \lambda = 0.
%\ln\mu & = & \sqrt{\lambda}\tfrac{i\mathbf{p}}{2} + \tfrac{1}{\sqrt{\lambda}}\int_0^{\mathbf{p}}\left(-i(\delbar u)^2 + \tfrac{i}{2}\cosh(2u)\right)\,dt + O(\lambda^{-1}) \;\text{ at } \lambda = \infty.
\end{eqnarray*}
\end{corollary}

\begin{proof}
From Theorem \ref{expansion} we know that at $\lambda = 0$ we have
\begin{eqnarray*}
\ln\mu & = & \tfrac{1}{\sqrt{\lambda}}\tfrac{i\mathbf{p}}{2} + \sqrt{\lambda}\int_0^{\mathbf{p}}b_1(t)\, dt + \sum_{m\geq 2} (\sqrt{\lambda})^m \int_0^{\mathbf{p}} b_m(t)\, dt \\
& = & \tfrac{1}{\sqrt{\lambda}}\tfrac{i\mathbf{p}}{2} + \sqrt{\lambda}\int_0^{\mathbf{p}}\left(-i(\del u)^2 + \tfrac{i}{2}\cosh(2u)\right)\,dt + O(\lambda).
\end{eqnarray*}
%The equation $M_{\lambda} = \left(\overline{M}_{\bar{\lambda}^{-1}}^t\right)^{-1}$ implies $\mu(\lambda) = \bar{\mu}^{-1}(\bar{\lambda}^{-1})$. Thus the expansion of $\ln\mu(\lambda)$ at $\lambda = \infty$ is equal to the expansion of $-\overline{\ln\mu(\bar{\lambda}^{-1})}$ at $\lambda = 0$ and one obtains
%$$
%\ln\mu = \sqrt{\lambda}\tfrac{i\mathbf{p}}{2} + \tfrac{1}{\sqrt{\lambda}}\int_0^{\mathbf{p}}\left(-i(\delbar u)^2 + \tfrac{i}{2}\cosh(2u)\right)\,dt + O(\lambda^{-1}) \;\text{ at } \lambda = \infty.
%$$
\end{proof}

\section{Polynomial Killing fields for finite type solutions}

In the following we will consider the variable $y$ as a flow parameter. 
%Expanding the matrices $U_{\lambda}$ and $V_{\lambda}$ with respect to this flow parameter $y$ we get for $U_\lambda$
%\begin{eqnarray*}
%U_\lambda(x,y) & = & U_{\lambda}(x,0) + y\delta U_{\lambda}(x) + O(y^2), \\
%\tfrac{d}{dx}U_{\lambda}(x,y) & = & \tfrac{d}{dx}U_{\lambda}(x,0) + y\tfrac{d}{dx}\delta U_{\lambda}(x) + O(y^2), \\
%\tfrac{d}{dy}U_{\lambda}(x,y) & = & \delta U_{\lambda}(x) + O(y)
%\end{eqnarray*}
%and for $V_{\lambda}$ the equations
%\begin{eqnarray*}
%V_{\lambda}(x,y) & = & V_{\lambda}(x,0) + y\delta V_{\lambda}(x) + O(y^2), \\
%\tfrac{d}{dx}V_{\lambda}(x,y) & = & \tfrac{d}{dx}V_{\lambda}(x,0) + y\tfrac{d}{dx}\delta V_{\lambda}(x) + O(y^2).
%\end{eqnarray*}
If we define the Lax operator $L_{\lambda}:= \tfrac{d}{dx} + U_{\lambda}$, the zero-curvature condition can be rewritten as Lax equation via
$$
\tfrac{d}{dy}U_{\lambda} - \tfrac{d}{dx}V_{\lambda} - [U_{\lambda},V_{\lambda}] = 0 \;\; \Longleftrightarrow \;\; \tfrac{d}{dy}L_{\lambda} = [L_{\lambda},V_{\lambda}].
$$
This corresponds to the $\sinh$-Gordon flow and there exist analogous equations for the "higher" flows. In the finite type situation there exists a linear combination of these higher flows such that the corresponding Lax equation is stationary, i.e. there exists a map $W_{\lambda}$ such that
$$
\tfrac{d}{dy}L_{\lambda} = [L_{\lambda},W_{\lambda}] = \tfrac{d}{dx}W_{\lambda} + [U_{\lambda},W_{\lambda}] \equiv 0.
$$
%we obtain the following equation with respect to the constant term $y = 0$
%$$
%\delta U_{\lambda}(x) - \tfrac{d}{dx}V_{\lambda}(x,0) - [U_{\lambda}(x,0),V_{\lambda}(x,0)] = 0
%$$
%and therefore with the Lax operator $L_{\lambda}(x) := \tfrac{d}{dx} + U_{\lambda}(x,0)$
%$$
%\delta L_{\lambda}(x) = \delta U_{\lambda}(x) =\tfrac{d}{dx}V_{\lambda}(x,0) + [U_{\lambda}(x,0),V_{\lambda}(x,0)] = [L_{\lambda}(x), V_{\lambda}(x,0)].
%$$
%If we replace $V_{\lambda}(x,0)$ by a map $W_{\lambda}(x)$ solving $\tfrac{d}{dx} W_{\lambda}(x) = [W_{\lambda}(x),U_{\lambda}(x,0)]$ we get for the constant term
%$$
%\delta L_{\lambda} = \delta U_{\lambda}(x,0) = [L_{\lambda},W_{\lambda}(x)] = \tfrac{d}{dx}W_{\lambda}(x) + [U_{\lambda}(x,0),W_{\lambda}(x)] \equiv 0,
%$$
%i.e. the corresponding Lax equation $\delta L_{\lambda} = [L_{\lambda},W_{\lambda}]$ is stationary. 
This leads to the following definition (cf. Definition 2.1 in \cite{Hauswirth_Kilian_Schmidt_1}).

\begin{definition}\label{killing1}
A pair $C^{\infty}(\RR/\mathbf{p} \ZZ) \times C^{\infty}(\RR/\mathbf{p} \ZZ) \ni (u,u_y) \simeq U_{\lambda}(\cdot,0)$ corresponding to a periodic solution of the $\sinh$-Gordon equation is of \textbf{finite type} if there exists $g \in \NN_0$ such that
$$
\Phi_{\lambda}(x) = \frac{\lambda^{-1}}{2}\begin{pmatrix} 0 & ie^u \\ 0 & 0 \end{pmatrix} + \sum_{n=0}^g \lambda^n
\begin{pmatrix} \omega_n & e^u\tau_n \\ e^u\sigma_n & -\omega_n \end{pmatrix}
$$
is a solution of the Lax equation
$$
\frac{d}{dx}\Phi_{\lambda} = [\Phi_{\lambda},U_{\lambda}(\cdot,0)]
$$
for some periodic functions $\omega_n, \tau_n, \sigma_n: \RR/\mathbf{p} \ZZ \to \CC$.
\end{definition}

\begin{remark}
Let us justify why we can set $y=0$ in the above definition:
\begin{itemize}
\item For a solution $u$ of the $\sinh$-Gordon equation \eqref{eq_sinh} on a strip around the $y=0$ axis, that is of finite type in the sense of Definition 2.1 in \cite{Hauswirth_Kilian_Schmidt_1}, we see that $(u(\cdot,0),u_y(\cdot,0)) \in C^{\infty}(\RR/\mathbf{p} \ZZ) \times C^{\infty}(\RR/\mathbf{p} \ZZ)$ is of finite type in the sense of Definition \ref{killing1}. 
\item On the other hand, every pair $(u,u_y) \in C^{\infty}(\RR/\mathbf{p} \ZZ) \times C^{\infty}(\RR/\mathbf{p} \ZZ)$ of finite type origins from a "global" finite type solution $u: \CC \to \RR$ in the sense of Definition 2.1 in \cite{Hauswirth_Kilian_Schmidt_1}.
\end{itemize}
\end{remark}

Given a map $\widetilde{\Phi}_{\lambda}$ of the form
$$
\widetilde{\Phi}_{\lambda}(z) = \frac{\lambda^{-1}}{2}\begin{pmatrix} 0 & ie^u \\ 0 & 0 \end{pmatrix} + \sum_{n=0}^g \lambda^n
\begin{pmatrix} \widetilde{\omega}_n & e^u\widetilde{\tau}_n \\ e^u\widetilde{\sigma}_n & -\widetilde{\omega}_n \end{pmatrix}
$$
%with expansion
%\begin{eqnarray*}
%\widetilde{\Phi}_{\lambda}(x,y) & = & \widetilde{\Phi}_{\lambda}(x,0) + y\delta \widetilde{\Phi}_{\lambda}(x) + O(y^2) \\
%\tfrac{d}{dx}\widetilde{\Phi}_{\lambda}(x,y) & = & \tfrac{d}{dx}\widetilde{\Phi}_{\lambda}(x,0) + y\tfrac{d}{dx}\delta \widetilde{\Phi}_{\lambda}(x) + O(y^2)\\
%\tfrac{d}{dy}\widetilde{\Phi}_{\lambda}(x,y) & = & \delta \widetilde{\Phi}_{\lambda}(x) + O(y)
%\end{eqnarray*}
that is a solution of the Lax equation (according to Definition 2.1 in \cite{Hauswirth_Kilian_Schmidt_1})
$$
 d\widetilde{\Phi}_{\lambda} = [\widetilde{\Phi}_{\lambda},\alpha_{\lambda}] \Longleftrightarrow 
\begin{cases} 
\tfrac{d}{dx}\widetilde{\Phi}_{\lambda}(x,y) = [\widetilde{\Phi}_{\lambda}(x,y),U_{\lambda}(x,y)]  \\
\tfrac{d}{dy}\widetilde{\Phi}_{\lambda}(x,y) = [\widetilde{\Phi}_{\lambda}(x,y),V_{\lambda}(x,y)] 
\end{cases}
$$
we obtain a map $\Phi_{\lambda}$ as in Definition \ref{killing1} by setting
$$
\Phi_{\lambda}(x) := \widetilde{\Phi}_{\lambda}(x,0).
$$
Let us omit the tilde in the following proposition.

\begin{proposition}[\cite{Hauswirth_Kilian_Schmidt_1}, Proposition 2.2]\label{prop_pinkall_sterling}
Suppose $\Phi_{\lambda}$ is of the form
$$
\Phi_{\lambda}(z) = \frac{\lambda^{-1}}{2}\begin{pmatrix} 0 & ie^u \\ 0 & 0 \end{pmatrix} + \sum_{n=0}^g \lambda^n
\begin{pmatrix} \omega_n & e^u\tau_n \\ e^u\sigma_n & -\omega_n \end{pmatrix}
$$
for some $u: \CC \to \RR$, and that $\Phi_{\lambda}$ solves the Lax equation $d\Phi_{\lambda} = [\Phi_{\lambda},\alpha_{\lambda}]$. Then
\begin{enumerate}
\item[(i)] The function $u$ is a solution of the $\sinh$-Gordon equation, i.e. $\Delta u + 2\sinh(2u) = 0$.
\item[(ii)] The functions $\omega_n$ are solutions of the homog. Jacobi equation $\Delta\omega_n + 4\cosh(2u)\omega_n = 0$.
\item[(iii)] The following iteration gives a formal solution of $d\Phi_{\lambda} = [\Phi_{\lambda},\alpha_{\lambda}]$. Let $\omega_n, \sigma_n,\tau_{n-1}$ with a solution $\omega_n$ of $\Delta\omega_n + 4\cosh(2u)\omega_n = 0$ be given. Now solve the system
$$
\tau_{n,\bar{z}} = ie^{-2u}\omega_n, \;\;\;\; \tau_{n,z} = 2iu_z\omega_{n,z} - i\omega_{n,zz}
$$
for $\tau_{n,z}$ and $\tau_{n,\bar{z}}$. Then define $\omega_{n+1}$ and $\sigma_{n+1}$ by
$$
\omega_{n+1} := -i\tau_{n,z} - 2i u_z \tau_n, \;\;\;\; \sigma_{n+1} := e^{2u}\tau_n + 2i\omega_{n+1,\bar{z}}.
$$
\item[(iv)] Each $\tau_n$ is defined up to a complex constant $c_n$, so $\omega_{n+1}$ is defined up to $-2ic_n u_z$.
\item[(v)] $\omega_0 = u_z, \omega_{g-1} = c u_{\bar{z}}$ for some $c \in \CC$, and $\lambda^g\overline{\Phi_{1/\bar{\lambda}}}^t$ also solves $d\Phi_{\lambda} = [\Phi_{\lambda},\alpha_{\lambda}]$.
\end{enumerate}
\end{proposition}

In \cite{Pinkall_Sterling} Pinkall-Sterling construct a series of solutions for the induction introduced in Proposition \ref{prop_pinkall_sterling}, (iii). From this \textbf{Pinkall-Sterling iteration} we obtain for the first terms of $\omega = \sum_{n\geq -1} \lambda^n \omega_n$
\begin{gather*}
\omega_{-1} = 0, \;\; \omega_0 = u_z = \tfrac{1}{2}(u_x - iu_y), \;\; \omega_1 = u_{zzz}-2(u_z)^3,\;\; \\
\omega_2 = u_{zzzzz} - 10u_{zzz}(u_z)^3 - 10(u_{zz})^2u_z + 6(u_z)^5,\; \ldots
\end{gather*}

\subsection{Potentials and polynomial Killing fields} 
We follow the exposition given in \cite{Hauswirth_Kilian_Schmidt_1}, Section 2. For $g \in \NN_0$ consider the $3g+1$-dimensional real vector space
\begin{gather*}
% \Lambda_{-1}^g\mathfrak{sl}(2,\CC) 
\Lambda_{-1}^g\mathfrak{sl}_2(\CC) = \bigg\{ \xi_{\lambda} = \sum_{n=-1}^g \lambda^n\widehat{\xi}_n \,\bigg|\, \widehat{\xi}_{-1} \in i\RR \epsilon_+, 
% \hspace{4cm} \\
% \hspace{4cm} 
\widehat{\xi}_n = -\overline{\widehat{\xi}_{g-1-n}}^t \in \mathfrak{sl}_2(\CC) \text{ for } n=-1,\ldots,g\bigg\}
\end{gather*}
and define an open subset of $\Lambda_{-1}^g\mathfrak{sl}_2(\CC)$ by
$$
\mathcal{P}_g := \{\xi_{\lambda} \in \Lambda_{-1}^g\mathfrak{sl}_2(\CC)\left|\right. \widehat{\xi}_{-1} \in i\RR^+ \epsilon_+, \trace(\widehat{\xi}_{-1}\widehat{\xi}_0) \neq 0\}.
$$
Every $\xi_{\lambda} \in \mathcal{P}_g$ satisfies the so-called \textbf{reality condition}
$$
\lambda^{g-1}\overline{\xi_{1/\bar{\lambda}}}^t = -\xi_{\lambda}.
$$

\begin{definition}
A \textbf{polynomial Killing field} is a map $\zeta_{\lambda}:\RR \to \mathcal{P}_g$ which solves
$$
\frac{d}{dx}\zeta_{\lambda} = [\zeta_{\lambda},U_{\lambda}(\cdot,0)] \;\;\; \text{ with } \;\;\; \zeta_{\lambda}(0) = \xi_{\lambda} \in \mathcal{P}_g.
$$
\end{definition}

For each initial value $\xi_{\lambda} \in \mathcal{P}_g$, there exists a unique polynomial Killing field given by
$$
\zeta_{\lambda}(x) := F_{\lambda}^{-1}(x)\xi_{\lambda} F_{\lambda}(x)
$$
with $\frac{d}{dx}F_{\lambda}(x) = F_{\lambda}(x)U_{\lambda}(x,0)$, since there holds
\begin{eqnarray*}
\tfrac{d}{dx}\zeta_{\lambda} & = & \tfrac{d}{dx}\left(F_{\lambda}^{-1}\xi_{\lambda} F_{\lambda}\right) = -F_{\lambda}^{-1}(\tfrac{d}{dx}F_{\lambda})F_{\lambda}^{-1}\xi_{\lambda} F_{\lambda} + F_{\lambda}^{-1}\xi_{\lambda} (\tfrac{d}{dx}F_{\lambda}) \\
& = & -U_{\lambda}(\cdot,0)F_{\lambda}^{-1}\xi_{\lambda} F_{\lambda} + F_{\lambda}^{-1}\xi_{\lambda} F_{\lambda}U_{\lambda}(\cdot,0) \\
& = & [\zeta_{\lambda},U_{\lambda}(\cdot,0)].
\end{eqnarray*}

In order to obtain a polynomial Killing field $\zeta_{\lambda}$ from a pair $(u,u_y)$ of finite type we set
$$
\zeta_{\lambda}(x) := \Phi_{\lambda}(x) - \lambda^{g-1}\overline{\Phi_{1/\bar{\lambda}}}^t(x) \;\;\; \text{ and } \;\;\; \zeta_{\lambda}(0) =: \xi_{\lambda} = \Phi_{\lambda}(0) - \lambda^{g-1}\overline{\Phi_{1/\bar{\lambda}}}^t(0).
$$

\begin{remark}
These polynomial Killing fields $\zeta_{\lambda}$ are \textit{periodic} maps $\zeta_{\lambda}:\RR/\mathbf{p}\ZZ \to \mathcal{P}_g$.
\end{remark}

Suppose we have a polynomial Killing field
$$
\zeta_{\lambda}(x) = \begin{pmatrix} 0 & \beta_{-1}(x) \\ 0 & 0 \end{pmatrix} \lambda^{-1} + \begin{pmatrix} \alpha_0(x) & \beta_{0}(x) \\ \gamma_0(x) & -\alpha_0(x) \end{pmatrix} \lambda^0 + \ldots + \begin{pmatrix} \alpha_g(x) & \beta_{g}(x) \\ \gamma_g(x) & -\alpha_g(x) \end{pmatrix} \lambda^g.
$$
Then one can assign a matrix-valued form $U(\zeta_{\lambda})$ to $\zeta_{\lambda}$ defined by
$$
U(\zeta_{\lambda})= 
\begin{pmatrix} 
\alpha_0(x) - \overline{\alpha}_0(x) & \lambda^{-1} \beta_{-1}(x) - \overline{\gamma}_0(x) \\
-\lambda\overline{\beta}_{-1}(x) + \gamma_0(x) & -\alpha_0(x) + \overline{\alpha}_0(x)
\end{pmatrix}dx.
$$

\begin{remark}
This shows that $(u,u_y) \simeq U_{\lambda}(\cdot,0)$ is uniquely defined by $\zeta_{\lambda}$.
\end{remark}

\section{The associated spectral data}

In this section we want to establish a 1:1-correspondence between pairs $(u,u_y)$ that originate from solutions of the $\sinh$-Gordon equation and the so-called spectral data $(Y(u,u_y),D(u,u_y))$ consisting of the spectral curve $Y(u,u_y)$ and a divisor $D(u,u_y)$ on $Y(u,u_y)$.

\subsection{Spectral curve defined by $\xi_{\lambda} \in \mathcal{P}_g$}

Let us introduce the definition of a spectral curve $Y$ that results from a periodic polynomial Killing field $\zeta_{\lambda}:\RR/\mathbf{p}\ZZ \to \mathcal{P}_g$.
\begin{definition}
Let $Y^*$ be defined by
$$
Y^* = \{(\lambda,\widetilde{\nu}) \in \CC^* \times \CC^* \left|\right. \det(\widetilde{\nu}\unity - \zeta_{\lambda}) = \widetilde{\nu}^2 + \det(\xi_{\lambda}) = 0\}
$$
and suppose that the polynomial $a(\lambda) = -\lambda \det(\xi_{\lambda})$ has $2g$ pairwise distinct roots. By declaring $\lambda=0,\infty$ to be two additional branch points and setting $\nu = \widetilde{\nu}\lambda$ one obtains that
$$
Y := \{(\lambda,\nu) \in \CC\PP^1 \times \CC\PP^1 \left|\right. \nu^2 = \lambda a(\lambda)\}
$$
defines a compact hyperelliptic curve $Y$ of genus $g$, the \textbf{spectral curve}. The genus $g$ of $Y$ is called the \textbf{spectral genus}.
\end{definition}

\begin{remark}
Note that the eigenvalue $\widetilde{\nu}$ of $\xi_{\lambda}$ is given by $\widetilde{\nu} = \frac{\nu}{\lambda}$.
\end{remark}

\subsection{Divisor}
The spectral curve $Y$ encodes the eigenvalues of $\xi_{\lambda} \in \mathcal{P}_g$. By considering the eigenvectors of $\xi_{\lambda}$ one arrives at the following lemma.

\begin{lemma}\label{lemma_eigenbundle}
On the spectral curve $Y$ there exist unique meromorphic maps $v(\lambda,\nu)$ and $w(\lambda,\nu)$ from $Y$ to $\CC^2$ such that
\begin{enumerate}
\item[(i)] For all $(\lambda,\nu) \in Y^*$ the value of $v(\lambda,\nu)$ is an eigenvector of $\xi_{\lambda}$ with eigenvalue $\nu$ and $w(\lambda,\nu)$ is an eigenvector of $\xi_{\lambda}^t$ with eigenvalue $\nu$, i.e. 
$$
\xi_{\lambda}v(\lambda,\nu) = \nu v(\lambda,\nu), \;\;\; \xi_{\lambda}^t w(\lambda,\nu) = \nu w(\lambda,\nu).
$$
\item[(ii)] The first component of $v(\lambda,\nu)$ and $w(\lambda,\nu)$ is equal to $1$, i.e. $v(\lambda,\nu) = (1, v_2(\lambda,\nu))^t$ and $w(\lambda,\nu) = (1, w_2(\lambda,\nu))^t$ on $Y$.
\end{enumerate}
\end{lemma}

The holomorphic map $v:Y\to\CC\PP^1$ from Lemma \ref{lemma_eigenbundle} motivates the following definition.

\begin{definition}\label{definition_eigen_bundle}
Set the \textbf{divisor} $D(u,u_y)$ as
$$
D(u,u_y) = -\big(v(\lambda,\nu)\big)
$$
and denote by $E(u,u_y)$ the holomorphic line bundle whose sheaf of holomorphic sections is given by $\mathcal{O}_{D(u,u_y)}$. Then $E(u,u_y)$ is called the \textbf{eigenline bundle}.
\end{definition}

\subsection{Characterization of spectral curves}

Given a hyperelliptic Riemann surface $Y$ with branch points over $\lambda=0$ ($y_0$) and $\lambda = \infty$ ($y_{\infty}$) we can deduce conditions such that $Y$ is the spectral curve of a periodic finite type solution of the $\sinh$-Gordon equation. \\

Let us recall the well-known characterization of such spectral curves (cf. \cite{Kilian_Schmidt_infinitesimal}, Section 1.2, in the case of immersed CMC tori in $\SS^3$). Note that  $M_{\lambda}$ and $\xi_{\lambda}$ commute, i.e. $[M_{\lambda},\xi_{\lambda}] = 0$. Thus it is possible to diagonalize them simultaneously and $\mu$ and $\nu$ can be considered as two different functions on the same Riemann surface $Y$.

\begin{theorem}[\cite{Kilian_Schmidt_infinitesimal}]\label{theorem_spectral_curve}
Let $Y$ be a hyperelliptic Riemann surface with branch points over $\lambda=0$ ($y_0$) and $\lambda = \infty$ ($y_{\infty}$). Then $Y$ is the spectral curve of a periodic real finite type solution of the $\sinh$-Gordon equation if and only if the following three conditions hold:
\begin{enumerate}
\item[(i)] Besides the hyperelliptic involution $\sigma$ the Riemann surface $Y$ has two further anti-holomorphic involutions $\eta$ and $\rho = \eta \circ \sigma$. Moreover, $\eta$ has no fixed points and $\eta(y_0) = y_{\infty}$.
\item[(ii)] There exists a non-zero holomorphic function $\mu$ on $Y\backslash\{y_0,y_{\infty}\}$ that obeys
$$
\sigma^* \mu = \mu^{-1},\;\; \eta^*\bar{\mu} = \mu,\;\; \rho^*\bar{\mu} = \mu^{-1}.
$$
\item[(iii)] The form $d\ln\mu$ is a meromorphic differential of the second kind with double poles at $y_0$ and $y_{\infty}$ only.
\end{enumerate}
\end{theorem}

%\begin{proof} We first consider the ``if''-part ``$\Rightarrow$'' and get the conditions (i) and (ii) from Remark \ref{involutions_with_nu} and Remark \ref{mu_holomorphic} together with Lemma \ref{involutions}. From Corollary \ref{expansion_mu} we also have
%\begin{eqnarray*}
%\ln\mu & = & \tfrac{1}{\sqrt{\lambda}}\tfrac{i\mathbf{p}}{2} + \sqrt{\lambda}\int_0^{\mathbf{p}}\left(-i(\del u)^2 + \tfrac{i}{2}\cosh(2u)\right)\,dt + O(\lambda) \; \text{ at } \lambda = 0,\\
%\ln\mu & = & \sqrt{\lambda}\tfrac{i\mathbf{p}}{2} + \tfrac{1}{\sqrt{\lambda}}\int_0^{\mathbf{p}}\left(-i(\delbar u)^2 + \tfrac{i}{2}\cosh(2u)\right)\,dt + O(\lambda^{-1}) \;\text{ at } \lambda = \infty
%\end{eqnarray*}
%and therefore get in the $\sqrt{\lambda}$-chart around $\lambda = 0$ and the $(1/\sqrt{\lambda})$-chart around $\lambda = \infty$
%\begin{eqnarray*}
%d\ln\mu & = & d\sqrt{\lambda}\left(-\tfrac{1}{\lambda}\tfrac{i\mathbf{p}}{2} + \int_0^{\mathbf{p}}\left(-i(\del u)^2 + \tfrac{i}{2}\cosh(2u)\right)\,dt + O(\sqrt{\lambda})\right) \; \text{ at } \lambda = 0,\\
%d\ln\mu & = & \frac{d}{\sqrt{\lambda}}\left(-\lambda\tfrac{i\mathbf{p}}{2} + \int_0^{\mathbf{p}}\left(-i(\delbar u)^2 + \tfrac{i}{2}\cosh(2u)\right)\,dt + O(1/\sqrt{\lambda})\right) \;\text{ at } \lambda = \infty.
%\end{eqnarray*}
%This implies condition (iii). The ``only if''-part ``$\Leftarrow$'' follows from Proposition \ref{one_to_one_correspondence} in Chapter 4.
%\end{proof}

%\begin{remark}
%Since $\sigma^*d\ln\mu = d\ln(1/\mu) = -d\ln\mu$, we see that $d\ln\mu$ changes its sign under the hyperelliptic involution $\sigma$.
%\end{remark}

Following the terminology of \cite{Hauswirth_Kilian_Schmidt_1, Kilian_Schmidt_cylinders, Kilian_Schmidt_infinitesimal}, we will describe spectral curves of periodic real finite type solutions of the $\sinh$-Gordon equation via hyperelliptic curves of the form
$$
\nu^2 = \lambda\, a(\lambda) = -\lambda^2\det(\xi_{\lambda}) = (\lambda \widetilde{\nu})^2.
$$
Here $\widetilde{\nu}$ is the eigenvalue of $\xi_{\lambda}$ and $\lambda: Y \to \CC\PP^1$ is chosen in a way such that $y_0$ and $y_{\infty}$ correspond to $\lambda = 0$ and $\lambda = \infty$ with 
$$
\sigma^*\lambda = \lambda,\;\; \eta^*\bar{\lambda} = \lambda^{-1},\;\; \rho^*\bar{\lambda} = \lambda^{-1}.
$$
% and
% $$
% a(\lambda) = (-1)^g \prod_{i=1}^g \frac{\bar{\alpha}_i}{|\alpha_i|}(\lambda - \alpha_i)(\lambda - \bar{\alpha}_i^{-1})
% $$
% with pairwise different branch points $\alpha_1,\ldots,\alpha_g \in \{\lambda \in \CC \left|\right. 0 < |\lambda| < 1\}$. Thus we get $\eta^*\overline{a} = a$ and $\rho^*\overline{a} = a$ and $\lambda^{-g}a(\lambda) > 0$ for all $\lambda \in \SS^1$. 
Note that the function $\lambda: Y \to \CC\PP^1$ is fixed only up to a M\"obius transformations of the form $\lambda \mapsto e^{2i\varphi}\lambda$. Moreover, $d\ln\mu$ is of the form
$$
d\ln\mu = \frac{b(\lambda)}{\nu}\frac{d\lambda}{\lambda},
$$
where $b$ is a polynomial of degree $g+1$ with $\lambda^{g+1}\overline{b(\bar{\lambda}^{-1})} = -b(\lambda)$.

\begin{definition}\label{definition_spectral_data}
The \textbf{spectral curve data} of a periodic real finite type solution of the $\sinh$-Gordon equation is a pair $(a,b) \in \CC^{2g}[\lambda] \times \CC^{g+1}[\lambda]$ such that
\begin{enumerate}
\item[(i)] $\lambda^{2g}\overline{a(\bar{\lambda}^{-1})} = a(\lambda)$ and $\lambda^{-g}a(\lambda) \leq 0$ for all $\lambda \in \SS^1$ and $|a(0)| = 1$.
\item[(ii)] On the hyperelliptic curve $\nu^2 = \lambda a(\lambda)$ there is a single-valued holomorphic function $\mu$ with essential singularities at $\lambda = 0$ and $\lambda = \infty$ with logarithmic differential
$$
d\ln\mu = \frac{b(\lambda)}{\nu}\frac{d\lambda}{\lambda}
$$
with $b(0) = i\frac{\sqrt{a(0)}}{2} \mathbf{p}$ that transforms under the three involutions
\begin{gather*}
 \sigma: (\lambda,\nu) \mapsto (\lambda,-\nu),\;\; \rho: (\lambda,\nu) \mapsto (\bar{\lambda}^{-1},-\bar{\lambda}^{-1-g}\bar{\nu}), \;\; 
 \eta: (\lambda,\nu) \mapsto (\bar{\lambda}^{-1},\bar{\lambda}^{-1-g}\bar{\nu})
\end{gather*}
according to $\sigma^*\mu = \mu^{-1},\, \rho^*\mu = \bar{\mu}^{-1}$ and $\eta^*\mu = \bar{\mu}$.
\end{enumerate}
\end{definition}

\begin{remark}\label{equivalent_spectral_data}
The conditions (i) and (ii) from Definition \ref{definition_spectral_data} are equivalent to the following conditions (cf. Definition 5.10 in \cite{Hauswirth_Kilian_Schmidt_1}):
\begin{enumerate}
\item[(i)] $\lambda^{2g}\overline{a(\bar{\lambda}^{-1})} = a(\lambda)$ and $\lambda^{-g}a(\lambda) \leq 0$ for all $\lambda \in \SS^1$ and $|a(0)| = 1$.
\item[(ii)] $\lambda^{g+1}\overline{b(\bar{\lambda}^{-1})} = -b(\lambda)$ and $b(0) = i\frac{\sqrt{a(0)}}{2} \mathbf{p}$.
\item[(iii)] $\int_{\alpha_i}^{1/\bar{\alpha}_i} \frac{b(\lambda)}{\nu}\frac{d\lambda}{\lambda} = 0$ for all roots $\alpha_i$ of $a$.
\item[(iv)] The unique function $h: \widetilde{Y} \to \CC$, where $\widetilde{Y} = Y \setminus\bigcup\gamma_i$ and $\gamma_i$ are closed cycles over the straight lines connecting $\alpha_i$ and $1/\bar{\alpha}_i$, obeying $\sigma^*h = -h$ and $dh=\frac{b(\lambda)}{\nu}\frac{d\lambda}{\lambda}$, satisfies $h(\alpha_i) \in \pi i \ZZ$ for all roots $\alpha_i$ of $a$.
\end{enumerate}
\end{remark}

\subsection{The moduli space}
Since a M\"obius transformation of the form $\lambda \mapsto e^{2i\varphi}\lambda$ changes the spectral curve data $(a,b)$ but does not change the corresponding periodic solution of the $\sinh$-Gordon equation we introduce the following definition.

\begin{definition}
For all $g\in \NN_0$ let $\mathcal{M}_g(\mathbf{p})$ be the space of equivalence classes of spectral curve data $(a,b)$ from Definition \ref{definition_spectral_data} with respect to the action of $\lambda \mapsto e^{2i\varphi}\lambda$ on $(a,b)$. $\mathcal{M}_g(\mathbf{p})$ is called the \textbf{moduli space} of spectral curve data for Cauchy data $(u,u_y)$ of periodic real finite type solutions of the $\sinh$-Gordon equation.
\end{definition}

Each pair of polynomials $(a,b) \in \mathcal{M}_g(\mathbf{p})$ represents a spectral curve $Y_{(a,b)}$ for Cauchy data $(u,u_y)$ of a periodic real finite type solution of the $\sinh$-Gordon equation. 
% The following definition will be important as well.
%  We will work with these ``coordinates'' in order to get a better understanding of the space $\mathcal{M}_g(\mathbf{p})$. For now we also state

\begin{definition}\label{definition_M_g_1}
Let 
\begin{eqnarray*}
\mathcal{M}_g^1(\mathbf{p}) & := & \{(a,b) \in \mathcal{M}_g(\mathbf{p}) \left|\right. a \text{ has } 2g \text{ pairwise distinct roots and  } \\
& & \hspace{2.9cm} (a,b) \text{ have no common roots}\}
\end{eqnarray*}
be the moduli space of non-degenerated smooth spectral curve data for Cauchy data $(u,u_y)$ of periodic real finite type solutions of the $\sinh$-Gordon equation.
\end{definition}

The term ``non-degenerated'' in Definition \ref{definition_M_g_1} reflects the following fact (cf. \cite{Hauswirth_Kilian_Schmidt_2}, Section 9): If one considers deformations of spectral curve data $(a,b)$, the corresponding integral curves have possible singularities, if $a$ and $b$ have common roots. By excluding the case of common roots of $(a,b)$, one can avoid that situation and identify the space of such deformations with certain polynomials $c \in \CC^{g+1}[\lambda]$ (see Section \ref{subsection_spectral_curve_deformation}).

\begin{remark}
By studying Cauchy data $(u,u_y)$ whose spectral curve $Y(u,u_y)$ corresponds to $(a,b) \in \mathcal{M}_g^1(\mathbf{p})$, we have the following benefits:
\begin{enumerate}
\item Since $(a,b) \in \mathcal{M}_g^1(\mathbf{p})$ correspond to Cauchy data $(u,u_y)$ of finite type, we can avoid difficult functional analytic methods for the asymptotic analysis of the spectral curves $Y$ at $\lambda = 0$ and $\lambda = \infty$. 
\item Since $(a,b) \in \mathcal{M}_g^1(\mathbf{p})$ have no common roots, we obtain non-singular smooth spectral curves $Y$ and can apply the standard tools from complex analysis for their investigation.
\end{enumerate}
Note, that these assumptions can be dropped in order to extend our results to the more general setting. This was done in \cite{Schmidt_infinite} for the case of the non-linear Schr\"odinger operator, for example.
\end{remark}

\subsection{Spectral data}
With all this terminology at hand let us introduce the spectral data associated to a periodic real finite type solution of the $\sinh$-Gordon equation.
\begin{definition}
The \textbf{spectral data} of a periodic real finite type solution of the $\sinh$-Gordon equation is a pair $(Y(u,u_y),D(u,u_y))$ such that $Y(u,u_y)$ is a hyperelliptic Riemann surface of genus $g$ that obeys the conditions from Theorem \ref{theorem_spectral_curve} and $D(u,u_y)$ is a divisor of degree $g+1$ on $Y(u,u_y)$ that obeys $\eta(D) - D = (f)$ for a meromorphic $f$ with $f\eta^*\bar{f} = -1$.
\end{definition}

From the investigation of the so-called "Inverse Problem" it is well-known that the correspondence between $(u,u_y)$ and the spectral data is bijective, since $(Y,D)$ uniquely determine $\xi_{\lambda}$ and $\zeta_{\lambda}$ respectively (cf. \cite{Hitchin_Harmonic}).

\section{Isospectral and non-isospectral deformations}

We will now consider deformations that correspond to variations of the divisor $D$ or the spectral curve $Y$ and call them isospectral and non-isospectral deformations, respectively.

\subsection{The projector $P$}

We will use the meromorphic maps $v:Y \to \CC^2$ and $w^t:Y \to \CC^2$ from Lemma \ref{lemma_eigenbundle} to define a matrix-valued meromorphic function on $Y$ by setting $P := \frac{v w^t}{w^t v}$. Given a meromorphic map $f$ on $Y$ we also define
$$
P(f) := \frac{v f w^t}{w^t v}.
$$
It turns out that $P$ is a projector and that it is possible to extend the definition of the projector $P$ to a projector $P_x$ that is defined by
$$
P_x(f) := \frac{v(x) f w(x)^t}{w(x)^t v(x)} = F_{\lambda}^{-1}(x)P(f) F_{\lambda}(x).
$$
Moreover, there holds $\zeta_{\lambda}(x) = P_x(\tfrac{\nu}{\lambda}) + \sigma^*P_x(\tfrac{\nu}{\lambda})$.

\subsection{General deformations of $M_{\lambda}$ and $U_{\lambda}$}

In the next lemma we consider the situation of a general variation with isospectral and non-isospectral parts.

\begin{lemma}\label{lemma_var_general}
Let $v_1, w_1^t$ be the eigenvectors for $\mu$ and $v_2, w_2^t$ the corresponding eigenvectors for $\tfrac{1}{\mu}$ of $M(\lambda)$. Then
\beq
\label{var_general}
\delta M(\lambda)v_1 + M(\lambda)\delta v_1 = (\delta\mu) v_1 + \mu \delta v_1 \;\text{ and }\; \delta M(\lambda)v_2 + M(\lambda)\delta v_2 = \delta(\tfrac{1}{\mu})v_2 +  \tfrac{1}{\mu} \delta v_2
\tag{$*$}
\ee
if and only if
$$
\label{var_general_M}
\delta M(\lambda) = \left[\sum_{i=1}^2 \frac{(\delta v_i) w_i^t}{w_i^t v_i},M(\lambda)\right] + (P(\delta\mu) +\sigma^*P(\delta\mu)).
%\tag{$**$}
$$
\end{lemma}

\begin{proof}
A direct calculation shows
\begin{small}
\begin{eqnarray*}
\left(\left[\sum_{i=1}^2 \frac{(\delta v_i) w_i^t}{w_i^t v_i},M(\lambda)\right]+ (P(\delta\mu) +\sigma^*P(\delta\mu))\right)v_1 
& = & \left(\sum_{i=1}^2 \frac{(\delta v_i) w_i^t}{w_i^t v_i}\right)\mu v_1 - M(\lambda)\delta v_1 +(\delta\mu) v_1 \\
& = & \mu \delta v_1 - M(\lambda)\delta v_1 +(\delta\mu) v_1 \stackrel{\eqref{var_general}}{=} \delta M(\lambda)v_1.
\end{eqnarray*}
\end{small}
An analogous calculation for $v_2$ gives
$$
\delta M(\lambda)v_2 = \left(\left[\sum_{i=1}^2 \frac{(\delta v_i) w_i^t}{w_i^t v_i},M(\lambda)\right]+ (P(\delta\mu) +\sigma^*P(\delta\mu))\right)v_2
$$
and the claim is proved.
\end{proof}

%Since the arguments from the previous proof carry over to the equation $M_{\lambda}(x)v(x) = \mu v(x)$ we get the following
%
%\begin{corollary}
%For the $x$-dependent monodromy $M_{\lambda}(x)$ a general variation is given by
%$$
%\delta M_{\lambda}(x) = \left[\sum_{i=1}^2 \frac{(\delta v_i(x)) w_i^t(x)}{w_i^t(x) v_i(x)},M_{\lambda}(x)\right] + (P_x(\delta\mu) +\sigma^*P_x(\delta\mu)).
%$$
%\end{corollary}

The above considerations also apply for the equation $L_{\lambda}(x) v(x) = (\frac{d}{dx}+U_{\lambda})v(x)= \frac{\ln\mu}{\mathbf{p}} \cdot v(x)$ around $\lambda = 0$ and yield the following lemma.

\begin{lemma}\label{general_variation_U}
Let $v_1(x), w_1^t(x)$ be the eigenvectors for $\mu$ and $v_2(x), w_2^t(x)$ the corresponding eigenvectors for $\tfrac{1}{\mu}$ of $M_{\lambda}(x)$ and $M_{\lambda}^t(x)$ respectively. Then
\begin{eqnarray*}
\label{var_L_general_1}
\delta U_{\lambda} v_1(x) + L_{\lambda}(x)\delta v_1(x) & = & (\tfrac{\delta\ln\mu}{\mathbf{p}}) v_1(x) + \tfrac{\ln\mu}{\mathbf{p}}\delta v_1(x), \\
%\tag{$*$}
%\label{var_L_general_2}
\delta U_{\lambda} v_2(x) + L_{\lambda}(x)\delta v_2(x) & = & -(\tfrac{\delta\ln\mu}{\mathbf{p}}) v_2(x) - \tfrac{\ln\mu}{\mathbf{p}}\delta v_2(x)
\end{eqnarray*}
around $\lambda = 0$ if and only if
\beq
\label{var_general_U}
\delta U_{\lambda} = \left[L_{\lambda}(x),-\sum_{i=1}^2 \frac{(\delta v_i(x)) w_i^t(x)}{w_i^t(x) v_i(x)}\right] + (P_x(\tfrac{\delta\ln\mu}{\mathbf{p}}) +\sigma^*P_x(\tfrac{\delta\ln\mu}{\mathbf{p}})).
%\tag{$**$}
\ee
\end{lemma}

\begin{proof}
Following the steps from the proof of Lemma \ref{lemma_var_general} and keeping in mind that $\delta \mathbf{p} = 0$ in our situation yields the claim.
\end{proof}

\begin{remark}
Equation \eqref{var_general_U} reflects the decomposition of the tangent space into isospectral and non-isospectral deformations.
\end{remark}

\subsection{Infinitesimal isospectral deformations of $\xi_{\lambda}$ and $U_{\lambda}$}

We will now investigate infinitesimal isospectral deformations. Note that there also exist global isospectral deformations that are described in \cite{Hauswirth_Kilian_Schmidt_1}.

\subsubsection{The real part of $H^1(Y,\mathcal{O})$}
Let us begin with the investigation of $H^1(Y,\mathcal{O})$, the Lie algebra of the Jacobian $\text{Jac}(Y)$. \\

For this, consider disjoint open simply-connected neighborhoods $U_0,U_{\infty}$ of $y_0,y_{\infty}$. Setting $U := Y\backslash\{y_0,y_{\infty}\}$ we get a cover $\mathcal{U} := \{U_0,U,U_{\infty}\}$ of $Y$. The only non-empty intersections of neighborhoods from $\mathcal{U}$ are $U_0\backslash\{y_0\}$ and $U_{\infty}\backslash\{y_{\infty}\}$. Since $U_0$ and $U_{\infty}$ are simply connected we have $H^1(U_0,\mathcal{O}) = 0 = H^1(U_{\infty},\mathcal{O})$. Moreover, $H^1(U,\mathcal{O}) = 0$ since $U$ is a non-compact Riemann surface. This shows that $\mathcal{U}$ is a Leray cover and therefore $H^1(Y,\mathcal{O}) = H^1(\mathcal{U},\mathcal{O})$, see \cite[Theorem 12.8]{Fo}. \\

We will consider triples of the form $(f_0,0,f_{\infty})$ of meromorphic functions with respect to this cover $\mathcal{U}$. Then the Cech-cohomology is induced by all pairs of the form $(f_0 - 0, f_{\infty}-0) = (f_0,f_{\infty})$.

\begin{lemma}\label{Lemma_ML}
The equivalence classes $[h_i]$ of the $g$ tuples $h_i := (f_0^i,f_{\infty}^i)$ given by $h_i = (\nu\lambda^{-i}, -\nu\lambda^{-i})$ for $i=1,\ldots,g$ are a basis of $H^1(Y,\mathcal{O})$. In particular $f_{\infty} = -f_0$.
\end{lemma}

\begin{proof}
We know that for $i=1,\ldots,g$ the differentials $\omega_i = \frac{\lambda^{i-1}d\lambda}{\nu}$ span a basis for $H^0(Y,\Omega)$ and that the pairing $\langle\cdot,\cdot\rangle: H^0(Y,\Omega)\times H^1(Y,\mathcal{O}) \to \CC$ given by $(\omega,[h]) \mapsto \text{Res}(h\omega)$ is non-degenerate due to Serre duality \cite[Theorem 17.9]{Fo}. Therefore we can calculate the dual basis of $\omega_i$ with respect to this pairing and see
\begin{eqnarray*}
\langle \omega_i,[h_j]\rangle & = & \text{Res}_{\lambda = 0} f_0^j\omega_i + \text{Res}_{\lambda = \infty} f_{\infty}^j\omega_i
 = \text{Res}_{\lambda = 0} \lambda^{i-j-1} d\lambda - \text{Res}_{\lambda = \infty} \lambda^{i-j-1} d\lambda \\
& = & \text{Res}_{\lambda = 0} \lambda^{i-j-1} d\lambda + \text{Res}_{\lambda = 0} \lambda^{-i+j+1} \frac{d\lambda}{\lambda^2}
 = \text{Res}_{\lambda = 0} (\lambda^{i-j-1} + \lambda^{-i+j-1}) d\lambda \\
& = & 2\cdot \delta_{ij}.
\end{eqnarray*}
This shows $\text{span}_{\CC}\{[h_1],\ldots,[h_g]\} = H^1(Y,\mathcal{O})$ and concludes the proof.
\end{proof}

\begin{definition}
Let $H^1_{\RR}(Y,\mathcal{O}) := \{[f] \in H^1(Y,\mathcal{O}) \left|\right. \overline{\eta^*f} = f\}$ be the real part of $H^1(Y,\mathcal{O})$ with respect to the involution $\eta$.
\end{definition}

\begin{lemma}\label{Lemma_reality_H1}
An element $[f] = [(f_0,f_{\infty})] \in H^1(Y,\mathcal{O})$ satisfies $\overline{\eta^*f} = f$ if and only if $f_{\infty} = \eta^*\bar{f}_0$. The corresponding $f_{\infty} = -f_0 = -\sum_{i=0}^{g-1} c_i \lambda^{-i-1} \nu$ satisfies $\bar{c}_i = -c_{g-1-i}$ for $i = 0,\ldots,g-1$. In particular $\dim_{\RR}H^1_{\RR}(Y,\mathcal{O}) = g$. Any element $[f] = [(f_0,\eta^*\bar{f}_0)] \in H^1_{\RR}(Y,\mathcal{O})$ can be represented by $f_0(\lambda,\nu)$ with
$$
f_0(\lambda,\nu) = \sum_{i=0}^{g-1} c_i \lambda^{-i-1} \nu 
% + \eta^*(\sum_{i=0}^{g-1} \overline{c_i \lambda^{-i-1} \nu})
$$
and $\bar{c}_i = -c_{g-1-i}$ for $i = 0,\ldots,g-1$.
\end{lemma}

\begin{proof}
% We peform a splitting of $f(\lambda,\nu)$ into the anti-invariant and invariant parts with respect to the involution $\eta$ and get
% \begin{footnotesize}
% $$
% f(\lambda,\nu) = \frac{1}{2}\left(\sum_{i=0}^{g-1} c_i \lambda^{-i-1} \nu - \eta^*(\sum_{i=0}^{g-1} \overline{c_i \lambda^{-i-1} \nu})\right) + \frac{1}{2}\left(\sum_{i=0}^{g-1} c_i \lambda^{-i-1} \nu + \eta^*(\sum_{i=0}^{g-1} \overline{c_i \lambda^{-i-1} \nu})\right).
% $$
% \end{footnotesize}
The first part of the lemma is obvious. Now a direct calculation gives
% Now $\eta^*[\bar{f}] = -[f]$ if and only the invariant part vanishes. 
\begin{footnotesize}
$$
-\eta^*(\sum_{i=0}^{g-1} \overline{c_i \lambda^{-i-1} \nu}) = -\sum_{i=0}^{g-1} \bar{c}_i \lambda^{i+1} \lambda^{-g-1}\nu = -\sum_{i=0}^{g-1} \bar{c}_i \lambda^{i-g} \nu 
 \stackrel{j := g-i-1}{=} -\sum_{j=0}^{g-1} \bar{c}_{g-1-j} \lambda^{-j-1} \nu = \sum_{i=0}^{g-1} c_i \lambda^{-i-1} \nu
$$
\end{footnotesize}
if and only if $\bar{c}_i = -c_{g-1-i}$ for $i = 0,\ldots,g-1$. The subspace of elements $(c_0,\ldots,c_{g-1}) \in \CC^g$ that obey these conditions is a real $g$-dimensional subspace.
\end{proof}

%\begin{corollary}\label{corollary_H1R}
%Any element $[f] = [(f_0,\eta^*\bar{f}_0)] \in H^1_{\RR}(Y,\mathcal{O})$ can be represented by $f_0(\lambda,\nu)$ with
%$$
%f_0(\lambda,\nu) = \sum_{i=0}^{g-1} c_i \lambda^{-i-1} \nu 
%% + \eta^*(\sum_{i=0}^{g-1} \overline{c_i \lambda^{-i-1} \nu})
%$$
%and $\bar{c}_i = -c_{g-1-i}$ for $i = 0,\ldots,g-1$.
%\end{corollary}

\subsubsection{The Krichever construction and the isospectral group action} 
The construction procedure for linear flows on $\text{Jac}(Y)$ is due to Krichever \cite{Krichever_algebraic}. In \cite{McIntosh_harmonic_tori} McIntosh desribes the Krichever construction for finite type solutions of the $\sinh$-Gordon equation. Let us consider the sequence
\beq
0 \to H^1(Y,\ZZ) \to H^1(Y,\mathcal{O}) \stackrel{\exp}{\to} H^1(Y,\mathcal{O}^*) \simeq \text{Pic}(Y) \stackrel{\deg}{\to} H^2(Y,\ZZ) \to 0,
\label{eq_sequence}
\ee
where $\text{Pic}(Y)$ is the Picard variety of $Y$.

\begin{definition}\label{def_quaternionic_divisor}
Let $\text{Pic}_d(Y)$ be the connected component of the Picard variety of $Y$ of divisors of degree $d$ and let
$$
\text{Pic}^{\RR}_d(Y) := \{D \in \text{Pic}_d(Y) \left|\right. \eta(D) - D = (f) \text{ for a merom. } f \text{ with } f\eta^*\bar{f} = -1\}
$$
be the set of \textbf{quaternionic divisors} of degree $d$ with respect to the involution $\eta$.
\end{definition}

Recall that $\text{Jac}(Y) \simeq \text{Pic}_0(Y)$. $\text{Pic}^{\RR}_{d}(Y)$ is also called the \textit{real part} of $\text{Pic}_{d}(Y)$. If we consider $H^1_{\RR}(Y,\mathcal{O})$ and restrict to the connected component $\text{Pic}^{\RR}_0(Y)$ of $\text{Pic}(Y)$, we get a map $L: \RR^g \simeq H^1_{\RR}(Y,\mathcal{O}) \to \text{Pic}^{\RR}_0(Y), (t_0,\ldots,t_{g-1}) \mapsto L(t_0,\ldots,t_{g-1})$ from \eqref{eq_sequence} and Lemma \ref{Lemma_reality_H1}. \\
It is well known that the divisor $D(u,u_y)$ from Definition \ref{definition_eigen_bundle} satisfies $D(u,u_y) \in \text{Pic}^{\RR}_{g+1}(Y)$ (see \cite{Hitchin_Harmonic}). One can also show  \cite[Theorem 5.17]{Knopf_phd} that the action of the tensor product on holomorphic line bundles induces a continuous commutative and transitive group action of $\RR^g$ on $\text{Pic}_{g+1}^{\RR}(Y)$, which is denoted by
$$
\pi: \RR^g \times \text{Pic}_{g+1}^{\RR}(Y) \to \text{Pic}_{g+1}^{\RR}(Y),\;\;\; ((t_0,\ldots,t_{g-1}),E) \mapsto \pi(t_0,\ldots,t_{g-1})(E)
$$
with
$$
\pi(t_0,\ldots,t_{g-1})(E) = E \otimes L(t_0,\ldots,t_{g-1}).
$$
Here $L(t_0,\ldots,t_{g-1})$ is the family in $\text{Pic}_0^{\RR}(Y)$ which is obtained by applying Krichever's construction procedure.

\subsubsection{Loop groups and the Iwasawa decomposition} 
For real $r \in (0,1]$, denote the circle with radius $r$ by $\SS_r = \{\lambda \in \CC \left|\right. |\lambda| = r\}$ and the open disk with boundary $\SS_r$ by $I_r = \{\lambda \in \CC \left|\right. |\lambda| < r\}$. Moreover, the open annulus with boundaries $\SS_r$ and $\SS_{1/r}$ is given by $A_r = \{\lambda \in \CC \left|\right. r <|\lambda| < 1/r\}$ for $r \in (0,1)$. For $r=1$ we set $A_1 := \SS^1$. The loop group $\Lambda_r SL(2,\CC)$ of $SL(2,\CC)$ is the infinite dimensional Lie group of analytic maps from $\SS_r$ to $SL(2,\CC)$, i.e.
$$
\Lambda_r SL(2,\CC) = \mathcal{O}(\SS_r,SL(2,\CC)).
$$
We will also need the following two subgroups of $\Lambda_r SL(2,\CC)$: First let
$$
\Lambda_r SU(2) = \{F \in \mathcal{O}(A_r,SL(2,\CC)) \left|\right. F|_{\SS^1} \in SU(2)\}.
$$
Thus we have
$$
F_{\lambda} \in \Lambda_r SU(2) \; \Longleftrightarrow \; \overline{F_{1/\bar{\lambda}}}^t = F_{\lambda}^{-1}.
$$
The second subgroup is given by
$$
\Lambda_r^+ SL(2,\CC) = \{B \in \mathcal{O}(I_r \cup \SS_r,SL(2,\CC)) \left|\right. B(0) = \left(\begin{smallmatrix} \rho & c \\ 0 & 1/\rho \end{smallmatrix}\right) \text{ for } \rho \in \RR^+ \text{ and } c \in \CC \}.
$$
The normalization $B(0) = B_0$ ensures that
$$
\Lambda_r SU(2) \cap \Lambda_r^+ SL(2,\CC) = \{\unity\}.
$$

Now we have the following important result due to Pressley-Segal \cite{Pressley_Segal} that has been generalized by McIntosh \cite{McIntosh_Toda}.

\begin{theorem}\label{theorem_iwasawa}
The multiplication $\Lambda_r SU(2) \times \Lambda_r^+ SL(2,\CC) \to \Lambda_r SL(2,\CC)$ is a surjective real analytic diffeomorphism. The unique splitting of an element $\phi_{\lambda} \in \Lambda_r SL(2,\CC)$ into 
$$
\phi_{\lambda} = F_{\lambda} B_{\lambda}
$$
with $F_{\lambda} \in \Lambda_r SU(2)$ and $B_{\lambda} \in \Lambda_r^+ SL(2,\CC)$ is called \textbf{r-Iwasawa decomposition} of $\phi_{\lambda}$ or \textbf{Iwasawa decomposition} if $r = 1$.
\end{theorem}

\begin{remark}
The r-Iwasawa decomposition also holds on the Lie algebra level, i.e. $\Lambda_r \mathfrak{sl}_2(\CC) = \Lambda_r \mathfrak{su}_2 \oplus \Lambda_r^+ \mathfrak{sl}_2(\CC)$. This decomposition will play a very important role in the following.
\end{remark}

In order to describe the infinitesimal isospectral deformations of $\xi_{\lambda}$ and $U_{\lambda}$, we follow the exposition of \cite{Audin}, IV.2.e. Let us transfer those methods to our situation.

\begin{theorem}\label{theorem_vf_isospectral_xi}
Let $f_0(\lambda,\nu) = \sum_{i=0}^{g-1} c_i \lambda^{-i-1} \nu$ be the representative of $[(f_0,\eta^*\bar{f}_0)] \in H^1_{\RR}(Y,\mathcal{O})$ from Lemma \ref{Lemma_reality_H1} and let
$$
A_{f_0} := P(f_0) + \sigma^*P(f_0) = \sum_{i=0}^{g-1} c_i \lambda^{-i} (P(\lambda^{-1}\nu) + \sigma^*P(\lambda^{-1}\nu)) = \sum_{i=0}^{g-1} c_i \lambda^{-i} \xi_{\lambda}
$$
be the induced element in $\Lambda_r\mathfrak{sl}_2(\CC)$. Then the vector field of the isospectral action $\pi: \RR^g \times \text{Pic}_{g+1}^{\RR}(Y) \to \text{Pic}_{g+1}^{\RR}(Y)$ at $E(u,u_y)$ takes the value
$$
\dot{\xi}_{\lambda} = [A_{f_0}^+,\xi_{\lambda}] = -[\xi_{\lambda},A_{f_0}^-].
$$
%under the equivariant map $E \mapsto \xi_{\lambda}(E)$ from Proposition \ref{proposition_equivariant_map}. 
Here $A_{f_0} = A_{f_0}^+ + A_{f_0}^-$ is the Lie algebra decomposition of the Iwasawa decomposition.
\end{theorem}

\begin{proof}
% Since $P(\cdot) = \sum_{i=1}^2 \frac{v_i (\cdot) w_i^t}{w_i^t v_i}$ 
We write $\nu$ for $\widetilde{\nu}$. Obviously there holds $A_{f_0} v = f_0 v$. If we write $\widetilde{v} = e^{\beta_0(t)}v$ for local sections $\widetilde{v}$ of $\mathcal{O}_D \otimes L(t_0,\ldots,t_{g-1})$ and $v$ of $\mathcal{O}_D$ with $\beta_0(t) = \sum_{i=0}^{g-1}c_i t_i \lambda^{-i-1} \nu$ we get
\begin{eqnarray*}
\delta \widetilde{v} & = & \dot{\beta}_0(t)\widetilde{v} + e^{\beta_0(t)}\delta v = f_0\widetilde{v} + e^{\beta_0(t)}\delta v \\
& = & A_{f_0} \widetilde{v} + e^{\beta_0(t)}\delta v \\
& = & e^{\beta_0(t)}(A_{f_0} v + \delta v).
\end{eqnarray*}
Moreover, $(f_0) \geq 0$ on $Y^*$. This shows that $A_{f_0} v +\delta v$ is a vector-valued section of $\mathcal{O}_{D}$ on $Y^*$ and defines a map $A_{f_0}^+$ such that
$$
A_{f_0} v + \delta v = A_{f_0}^+ v
$$
% Since $v_1 = 1$ and $(\delta v)_1 = 0$ we get $f_0 = (A_{f_0}^+ v)_1$. 
holds. Since $A_{f_0} v = f_0 v$ we also obtain the equations
$$
\begin{cases}
\xi_{\lambda} v = \nu v, \\
\xi_{\lambda}(A_{f_0}^+ v -\delta v) = \nu(A_{f_0}^+ v -\delta v).
\end{cases}
$$
This implies
$$
\xi_{\lambda}\delta v + [A_{f_0}^+,\xi_{\lambda}]v = \nu \delta v.
$$
Differentiating the equation $\xi_{\lambda} v = \nu v$ we additionally obtain
$$
\dot{\xi}_{\lambda} v + \xi_{\lambda}\delta v = \dot{\nu}v + \nu\delta v = \nu\delta v.
$$
Combining the last two equations yields
$$
\dot{\xi}_{\lambda} v = [A_{f_0}^+,\xi_{\lambda}]v
$$
and concludes the proof since this equation holds for a basis of eigenvectors.
\end{proof}

\begin{remark}\label{remark_Af_minus}
The decomposition of $A_f \in \Lambda_r\mathfrak{sl}(2,\CC) = \Lambda_r \mathfrak{su}_2 \oplus \Lambda^+_{r}\mathfrak{sl}_2(\CC)$ yields $A_{f_0} = A_{f_0}^+ + A_{f_0}^-$ and therefore $A_{f_0} v + \delta v = A_{f_0}^+ v$ implies
$$
\delta v = -A_{f_0}^- v.
$$
In particular $A_{f_0}^-$ is given by $A_{f_0}^- = -\sum \frac{\delta v w^t}{w^t v}$.
%  and consequently
% $$
% A_{f_0}^+ = \sum \frac{v f w^t}{w^t v} + \sum \frac{\delta v w^t}{w^t v}.
% $$
\end{remark}

We want to extend Theorem \ref{theorem_vf_isospectral_xi} to obtain the value of the vector field induced by $\pi: \RR^g \times \text{Pic}_{g+1}^{\RR}(Y) \to \text{Pic}_{g+1}^{\RR}(Y)$ for $U_{\lambda}$.

\begin{theorem}\label{theorem_vf_isospectral_U}
Let $f_0(\lambda,\nu) = \sum_{i=0}^{g-1} c_i \lambda^{-i-1} \nu$ be the representative of $[(f_0,\eta^*\bar{f}_0)] \in H^1_{\RR}(Y,\mathcal{O})$ from Lemma \ref{Lemma_reality_H1} and let
$$
A_{f_0}(x) := P_x(f_0) + \sigma^*P_x(f_0) = \sum_{i=0}^{g-1} c_i \lambda^{-i} (P_x(\lambda^{-1}\nu) + \sigma^*P_x(\lambda^{-1}\nu)) = \sum_{i=0}^{g-1} c_i \lambda^{-i} \zeta_{\lambda}(x)
$$
be the induced map $A_{f_0}: \RR \to \Lambda_r\mathfrak{sl}_2(\CC)$. Then the vector field of the isospectral action $\pi: \RR^g \times \text{Pic}_{g+1}^{\RR}(Y) \to \text{Pic}_{g+1}^{\RR}(Y)$ at $E(u,u_y)$ takes the value
$$
\delta U_{\lambda}(x) = [A_{f_0}^+(x),L_{\lambda}(x)] = [L_{\lambda}(x),A_{f_0}^-(x)].
$$
Here $A_{f_0}(x) = A_{f_0}^+(x) + A_{f_0}^-(x)$ is the Lie algebra decomposition of the Iwasawa decomposition.
\end{theorem}

\begin{proof}
Obviously $A_{f_0}(x) v(x) = f_0 v(x)$. In analogy to the proof of Theorem \ref{theorem_vf_isospectral_xi} we obtain a map $A_{f_0}^+(x)$ such that
$$
A_{f_0}(x) v(x) + \delta v(x) = A_{f_0}^+(x) v(x)
$$
holds. 
% Since $v_1(x) = 1$ and $(\delta v(x))_1 = 0$ we get $f_0 = (A_{f_0}^+(x) v(x))_1$. 
Since $A_{f_0}(x) v(x) = f_0 v(x)$ we also obtain the equations
$$
\begin{cases}
L_{\lambda}(x) v(x) = \tfrac{\ln\mu}{\mathbf{p}} v(x), \\
L_{\lambda}(x) (A_{f_0}^+(x) v(x) -\delta v(x)) = \tfrac{\ln\mu}{\mathbf{p}}(A_{f_0}^+(x) v(x) -\delta v(x))
\end{cases}
$$
around $\lambda = 0$. This implies
$$
L_{\lambda}(x)\delta v(x) + [A_{f_0}^+(x),L_{\lambda}(x)]v(x) = \tfrac{\ln\mu}{\mathbf{p}} \delta v(x).
$$
Differentiating the equation $L_{\lambda}(x) v(x) = \tfrac{\ln\mu}{\mathbf{p}} v(x)$ we additionally obtain
$$
\delta L_{\lambda}(x) v(x) + L_{\lambda}(x) \delta v(x) = \tfrac{\delta\ln\mu}{\mathbf{p}} v(x) + \tfrac{\ln\mu}{\mathbf{p}} \delta v(x) = \tfrac{\ln\mu}{\mathbf{p}} \delta v(x).
$$
Combining the last two equations yields
$$
\delta L_{\lambda}(x) v(x) = \delta U_{\lambda}(x) v(x) = [A_{f_0}^+(x),L_{\lambda}(x)]v(x).
$$
\end{proof}

\subsection{Infinitesimal deformations of spectral curves}
\label{subsection_spectral_curve_deformation}
% We want to establish a connection between periodic Cauchy data $(u,u_y)$ of real finite type solutions of the $\sinh$-Gordon equation and spectral curves $Y$ with the properties described in Chapter 3.

\begin{definition}
Let $\Sigma_g^{\mathbf{p}}$ denote the space of smooth hyperelliptic Riemann surfaces $Y$ of genus $g$ with the properties described in Theorem \ref{theorem_spectral_curve}, such that $d\ln\mu$ has no roots at the branchpoints of $Y$.
\end{definition}

We will now investigate deformations of $\Sigma_g^{\mathbf{p}} \simeq \mathcal{M}^1_g(\mathbf{p})$. For this, following the expositions of \cite{Grinevich_Schmidt, Hauswirth_Kilian_Schmidt_2, Kilian_Schmidt_cylinders}, we derive vector fields on open subsets of $\mathcal{M}^1_g(\mathbf{p})$ and parametrize the corresponding deformations by a parameter $t \in [0,\varepsilon)$. We consider deformations of $Y(u,u_y)$ that preserve the periods of $d\ln\mu$. We already know that
\begin{gather*}
\int_{a_i} d\ln\mu = 0 \;\; \text{ and } \;\; \int_{b_i} d\ln\mu \in 2\pi i \ZZ \;\; \text{ for } i = 1,\ldots,g.
\end{gather*}
Considering the Taylor expansion of $\ln\mu$ with respect to $t$ we get
$$
\ln\mu(t) = \ln\mu(0) + t\, \del_t\ln\mu(0) + O(t^2)
$$
and thus
$$
d \ln\mu(t) = d \ln\mu(0) + t\, d \del_t \ln\mu(0) + O(t^2).
$$
Given a closed cycle $\gamma \in H_1(Y,\ZZ)$ we have
$$
\int_{\gamma} d \ln\mu(t) = \int_{\gamma} d \ln\mu(0) + t \int_{\gamma} d \del_t \ln\mu(0) + O(t^2)
$$
and therefore
$$
\frac{d}{dt}\left(\int_{\gamma} d \ln\mu(t)\right)\bigg|_{t=0} = \int_{\gamma} d \del_t \ln\mu(0) = 0.
$$
This shows that deformations resulting from the prescription of $\del_t \ln\mu|_{t=0}$ at $t=0$ preserve the periods of $d\ln\mu$ infinitesimally along the deformation and thus are \textit{isoperiodic}. If we set $\dot{\lambda} = 0$ we can consider $\ln\mu$ as a function of $\lambda$ and $t$ and get
$$
\ln\mu(\lambda,t) = \ln\mu(\lambda,0) + t\, \del_t\ln\mu(\lambda,0) + O(t^2)
$$
as well as
$$
 \del_{\lambda}\ln\mu(\lambda,t) = \del_{\lambda}\ln\mu(\lambda,0) + t\, \del^2_{\lambda\,t} \ln\mu(\lambda,0) + O(t^2).
$$
%We know that $\widehat{Y} = \{(\lambda,\mu) \in \CC^*\times\CC^* \left|\right. R(\lambda,\mu) = \mu^2 - \Delta(\lambda)\mu + 1 = 0\}$. Differentiating the expression $R(\lambda,\mu)= 0$ with respect to $\lambda$ we get
%$$
%2\mu\del_{\lambda}\mu - \Delta'(\lambda)\mu - \Delta(\lambda)\del_{\lambda}\mu = 2\mu^2\frac{\del_{\lambda}\mu}{\mu} - \Delta'(\lambda)\mu - \mu\Delta(\lambda)\frac{\del_{\lambda}\mu}{\mu} = 0
%$$
%and therefore
%$$
%\del_{\lambda}\ln\mu =  \frac{\Delta'(\lambda)\mu}{2\mu^2 - \mu\Delta(\lambda)} =  \frac{\Delta'(\lambda)}{2\mu - \Delta(\lambda)}
%$$
%on $\widehat{Y}$. 
On the compact spectral curve $Y$ the function $\del_{\lambda}\ln\mu$ is given by
$$
\del_{\lambda}\ln\mu = \frac{b(\lambda)}{\lambda \, \nu}
$$
and the compatibility condition $\del^2_{t\, \lambda} \ln\mu|_{t=0} = \del^2_{\lambda\, t} \ln\mu|_{t=0}$ will lead to a deformation of the spectral data $(a,b)$ or equivalently to a deformation of the spectral curve $Y$ with its differential $d\ln\mu$. Therefore this deformation is \textit{non-isospectral}. In the following we will investigate which conditions $\delta\ln\mu := (\del_t \ln\mu)|_{t=0}$ has to obey in order obtain such a deformation. \\
% From Theorem \ref{theorem_tangentspace_spectral_curve} we know that $T_{Y(u,u_y)}\mathcal{M}_g \simeq H^0(Y(u,u_y),\Omega)$.

\noindent{\bf The Whitham deformation.} Following the ansatz given in \cite{Hauswirth_Kilian_Schmidt_2} we consider the function $\ln\mu$ as a function of $\lambda$ and $t$ and write $\ln\mu$ locally as
$$
\ln\mu =
\begin{cases}
f_{\alpha_i}(\lambda)\sqrt{\lambda-\alpha_i} + \pi i n_i & \text{at a zero } \alpha_i \text{ of } a, \\
f_0(\lambda)\lambda^{-1/2} + \pi i n_0 & \text{at } \lambda = 0, \\
f_{\infty}(\lambda)\lambda^{1/2} + \pi i n_{\infty} & \text{at } \lambda = \infty.
\end{cases}
$$
Here we choose small neighborhoods around the branch points such that each neighborhood contains at most one branch point. Moreover, the functions $f_{\alpha_i},f_0,f_{\infty}$ do not vanish at the corresponding branch points. If we write $\dot{g}$ for $(\del_t g)|_{t=0}$ we get
$$
(\del_t\ln\mu)|_{t=0} =
\begin{cases}
 \dot{f}_{\alpha_i}(\lambda)\sqrt{\lambda-\alpha_i} - \frac{\dot{\alpha_i}f_{\alpha_i}(\lambda)}{2\sqrt{\lambda - \alpha_i}} & \text{at a zero } \alpha_i \text{ of } a, \\
\dot{f}_0(\lambda) \lambda^{-1/2}  & \text{at } \lambda = 0, \\
\dot{f}_{\infty}(\lambda)\lambda^{1/2}  & \text{at } \lambda = \infty.
\end{cases}
$$
Since the branches of $\ln\mu$ differ from each other by an element in $2\pi i \ZZ$ we see that $\delta\ln\mu = (\del_t \ln\mu)|_{t=0}$ is a single-valued meromorphic function on $Y$ with poles at the branch points of $Y$, i.e. the poles of $\delta\ln\mu$ are located at the zeros of $a$ and at $\lambda = 0$ and $\lambda = \infty$. Thus we have
$$
\delta\ln\mu = \frac{c(\lambda)}{\nu}
$$
with a polynomial $c$ of degree at most $g+1$. Since $\eta^*\delta\ln\mu = \delta\ln\bar{\mu}$ and $\eta^*\nu = \bar{\lambda}^{-g-1}\bar{\nu}$ the polynomial $c$ obeys the reality condition
\beq
\lambda^{g+1}\overline{c(\bar{\lambda}^{-1})} = c(\lambda).
\label{reality_condition_c}
\ee
Differentiating $\nu^2 = \lambda a$ with respect to $t$ we get $2\nu\dot{\nu} = \lambda \dot{a}$. The same computation for the derivative with respect to $\lambda$ gives $2\nu\nu' = a + \lambda a'$. Now a direct calculation shows
\begin{eqnarray*}
\del^2_{t\, \lambda} \ln\mu|_{t=0} & = & \del_t \left(\frac{b}{\lambda\, \nu}\right)\bigg|_{t=0} = \frac{\dot{b}\lambda\nu - b\lambda\dot{\nu}}{\lambda^2\nu^2} = \frac{2\dot{b}a-b\dot{a}}{2\nu^3}, \\
\del^2_{\lambda\, t} \ln\mu|_{t=0} & = & \del_{\lambda}\left(\frac{c}{\nu}\right) = \frac{c'\nu - c\nu'}{\nu^2} = \frac{2c'\nu^2 - 2c\nu\nu'}{2\nu^3} = \frac{2c'\lambda a - ca -c\lambda a'}{2\nu^3}.
\end{eqnarray*}
The compatibility condition $\del^2_{t\, \lambda} \ln\mu|_{t=0} = \del^2_{\lambda\, t} \ln\mu|_{t=0}$ holds if and only if
\beq
-2\dot{b}a + b\dot{a} = -2\lambda ac' + ac + \lambda a'c.
\label{compatibility_deformations}
\ee
Equation \eqref{compatibility_deformations} is the so-called \textbf{Whitham equation}. Both sides of this equation are polynomials of degree at most $3g+1$ and therefore describe relations for $3g+2$ coefficients. If we choose a polynomial $c$ that obeys the reality condition \eqref{reality_condition_c} we obtain a vector field on $\mathcal{M}^1_g(\mathbf{p})$. Since $(a,b) \in \mathcal{M}^1_g(\mathbf{p})$ have no common roots, the polynomials $a,b,c$ in equation \eqref{compatibility_deformations} uniquely define a tangent vector $(\dot{a},\dot{b})$ (see \cite{Hauswirth_Kilian_Schmidt_2}, Section 9).
%An application of these techniques to study CMC tori in $\SS^3$ and $\HH^3$ can be found in \cite{Slawa} and \cite{matthias}. 
In the following we will specify such polynomials $c$ that lead to deformations which do not change the period $\mathbf{p}$ of $(u,u_y)$ (cf. \cite{Hauswirth_Kilian_Schmidt_2}, Section 9).\\

\noindent{\bf Preserving the period $\mathbf{p}$ along the deformation.} If we evaluate the compatibility equation \eqref{compatibility_deformations} at $\lambda = 0$ we get
$$
-2\dot{b}(0)a(0) + b(0)\dot{a}(0) = a(0)c(0).
$$
Moreover,
$$
\dot{\mathbf{p}} = 2 \frac{d}{dt}\left(\frac{b(0)}{i\sqrt{a(0)}}\right)\bigg|_{t=0} = \frac{-2\dot{b}(0)a(0) + b(0)\dot{a}(0)}{-i (a(0))^{3/2}} = i\frac{c(0)}{\sqrt{a(0)}}.
$$
This proves the following lemma.

\begin{lemma}\label{lemma_fixed_period}
Vector fields on $\mathcal{M}^1_g(\mathbf{p})$ that are induced by polynomials $c$ obeying \eqref{reality_condition_c} preserve the period $\mathbf{p}$ of $(u,u_y)$ if and only if $c(0) = 0$.
\end{lemma}

Let us take a closer look at the space of polynomials $c$ that induce a Whitham deformation.

\begin{lemma}\label{lemma_coefficients_c}
For the coefficients of the polynomial $c(\lambda) = \sum_{i=0}^{g+1} c_i \lambda^i$ obeying \eqref{reality_condition_c} there holds $c_i = \bar{c}_{g+1-i}$ for $i=0,\ldots,g+1$.
\end{lemma}

%\begin{proof}
%Inserting $c(\lambda) = \sum_{i=0}^{g+1} c_i \lambda^i$ into \eqref{reality_condition_c} yields
%$$
%\lambda^{g+1} \sum_{i=0}^{g+1} \bar{c}_i \lambda^{-i} = \sum_{i=0}^{g+1} \bar{c}_i \lambda^{g+1-i} = \sum_{i=0}^{g+1} \bar{c}_{g+1-i} \lambda^{i} \stackrel{!}{=} \sum_{i=0}^{g+1} c_i \lambda^i.
%$$
%Equating the coefficients shows $c_i = \bar{c}_{g+1-i}$ for $i=0,\ldots,g+1$ and concludes the proof.
%\end{proof}

\begin{proposition}\label{dimension_moduli_space}
The space of polynomials c corresponding to deformations of spectral curve data $(a,b) \in \mathcal{M}^1_g(\mathbf{p})$ (with fixed period $\mathbf{p}$) is $g$-dimensional. 
\end{proposition}

\begin{proof}
The space of polynomials $c$ of degree at most $g+1$ obeying the reality condition \eqref{reality_condition_c} is $(g+2)$-dimensional. From Lemma \ref{lemma_fixed_period} we know that $\dot{\mathbf{p}} = 0$ if and only if $c(0) = 0$. This yields the claim.
\end{proof}

\subsection{$\mathcal{M}_g^1(\mathbf{p})$ is a smooth $g$-dimensional manifold}

From Proposition \ref{dimension_moduli_space} we know that the space of polynomials $c$ corresponding to deformations of $\mathcal{M}^1_g(\mathbf{p})$ with fixed period $\mathbf{p}$ is $g$-dimensional. In the following we want to show that $\mathcal{M}^1_g(\mathbf{p})$ is a real $g$-dimensional manifold. Therefore, we follow the terminology introduced by Carberry and Schmidt in \cite{Carberry_Schmidt}. \\

Let us recall the conditions that characterize a representative $(a,b) \in \CC^{2g}[\lambda]\times \CC^{g+1}[\lambda]$ of an element in $\mathcal{M}_g(\mathbf{p})$:
\begin{enumerate}
\item[(i)] $\lambda^{2g}\overline{a(\bar{\lambda}^{-1})} = a(\lambda)$ and $\lambda^{-g}a(\lambda) < 0$ for all $\lambda \in \SS^1$ and $|a(0)| = 1$.
\item[(ii)] $\lambda^{g+1}\overline{b(\bar{\lambda}^{-1})} = -b(\lambda)$.
\item[(iii)] $f_i(a,b) := \int_{\alpha_i}^{1/\bar{\alpha}_i} \frac{b}{\nu} \frac{d\lambda}{\lambda} = 0$ for the roots $\alpha_i$ of $a$ in the open unit disk $\DD \subset \CC$.
\item[(iv)] The unique function $h: \widetilde{Y} \to \CC$ with $\sigma^*h = -h$ and $dh=\frac{b}{\nu} \frac{d\lambda}{\lambda}$ satisfies $h(\alpha_i) \in \pi i \ZZ$ for all roots $\alpha_i$ of $a$.
\end{enumerate}

\begin{definition}
Let $\mathcal{H}^g$ be the set of polynomials $a \in \CC^{2g}[\lambda]$ that satisfy condition (i) and whose roots are pairwise distinct.
\end{definition}

Every $a \in \mathcal{H}^g$ corresponds to a smooth spectral curve.
% , i.e. a finite-type solution without bubbletons. 
Moreover, every $a \in \mathcal{H}^g$ is uniquely determined by its roots.

\begin{definition}
For every $a\in \mathcal{H}^g$ let the space $\mathcal{B}_a$ be given by
$$
\mathcal{B}_a := \{b \in \CC^{g+1}[\lambda] \left|\right. b \text{ satisfies conditions (ii) and (iii)}\}.
$$
\end{definition}

Since (iii) imposes $g$ linearly independent constraints on the $(g+2)$-dimensional space of polynomials $b \in \CC^{g+1}[\lambda]$ obeying the reality condition (ii) we get

\begin{proposition}[\cite{Carberry_Schmidt}]\label{prop_dim_B_a}
$\dim_{\RR} \mathcal{B}_a = 2$. In particular every $b_0 \in \CC$ uniquely determines an element $b \in \mathcal{B}_a$ with $b(0) = b_0$.
\end{proposition}

% \begin{remark}
% Proposition \ref{prop_dim_B_a} shows that $d\ln\mu$ is uniquely determined by $Y$ (represented by $\nu^2 = \lambda a(\lambda)$) in case of a fixed period $\mathbf{p}$ of $(u,u_y)$.
% \end{remark}

Using the Implicit Function Theorem one obtains the following proposition.

\begin{proposition}
The set 
$$
M:= \{(a,b) \in \CC^{2g}[\lambda]\times \CC^{g+1}[\lambda] \left|\right. a \in \mathcal{H}^g, (a,b) \text{ have no common roots and } b \text{ satisfies (ii)}\}
$$
is an open subset of a $(3g+2)$-dimensional real vector space. Moreover, the set 
$$
N:= \{(a,b) \in M \left|\right. f_i(a,b) = 0 \text{ for } i=1,\ldots,g\}
$$
defines a real submanifold of $M$ of dimension $2g+2$ that is parameterized by $(a,b(0))$. If $b(0) = b_0$ is fixed we get a real submanifold of dimension $2g$.
\end{proposition}

%\begin{proof}
%Consider the map $f =(f_1,\ldots,f_g): \RR^{2g+2}\times \RR^g \to \RR^g$ given by
%$$
%((a,b_0),(b_1,\ldots,b_{(g+1)/2})) \mapsto f(a,b) := (f_1(a,b),\ldots,f_g(a,b)).
%$$
%If we choose $b \in \mathcal{B}_a$ with $b(0) = b_0$ we get $f(a,b) = 0$ due to Proposition \ref{prop_dim_B_a}. Moreover, $f$ is linear with respect to $(b_1,\ldots,b_{(g+1)/2})$ and thus $\frac{\del(f_1,\ldots,f_g)}{\del(b_1,\ldots,b_{(g+1)/2})}$ is invertible at $(a,b)$. Now we can apply the Implicit Function Theorem and see that there exist neighborhoods $U \subset \RR^{2g+2}$ and $V \subset \RR^g$ with $(a,b_0) \in U$ and $(b_1,\ldots,b_{(g+1)/2}) \in V$ and a smooth map $g: U \to V$ with $g(a,b_0) = (b_1,\ldots,b_{(g+1)/2})$ such that 
%$$
%f((a,b_0),g(a,b_0)) = 0 \text{ for all } (a,b_0) \in U.
%$$
%Therefore $N = f^{-1}[0]$ defines a real submanifold of $M$ of dimension $2g+2$ that is parameterized by $(a,b(0))$.
%\end{proof}

The results in \cite{Hauswirth_Kilian_Schmidt_1, Hauswirth_Kilian_Schmidt_2} yield the following theorem (cf. Lemma 5.3 in \cite{Hauswirth_Kilian_Schmidt_3}).

\begin{theorem}\label{theorem_Mg}
For a fixed choice $n_1,\ldots,n_g \in \ZZ$ the map $h = (h_1,\ldots,h_g): N \to (i\RR/2\pi i \ZZ)^g \simeq (\SS^1)^g$ with
$$
h_j: N \to i\RR/2\pi i \ZZ, \;\; (a,b) \mapsto h_j(a,b):= \ln\mu(\alpha_j) - \pi i n_j
$$
is smooth and its differential $dh$ has full rank. In particular $\mathcal{M}^1_g(\mathbf{p}) = h^{-1}[0]$ defines a real submanifold of dimension $g$. Here we consider $b(0) = b_0$ as fixed, i.e. $\dim_{\RR}(N) = 2g$.
\end{theorem}

\begin{proof}
Let us consider an integral curve $(a(t),b(t))$ for the vector field $X_c$ that corresponds to a Whitham deformation that is induced by a polynomial $c$ obeying the reality condition \eqref{reality_condition_c}. Then there holds $h(a(t),b(t)) \equiv \text{const.}$ along this deformation and therefore
$$
dh(a(t),b(t))\cdot (\dot{a}(t),\dot{b}(t)) = dh(a(t),b(t))\cdot X_c(a(t),b(t)) = 0.
$$
From Proposition \ref{dimension_moduli_space} we know that the space of polynomials $c$ that correspond to a deformation with fixed period $\mathbf{p}$ is $g$-dimensional. We will now show that the map
\beq
c \mapsto X_c(a(t),b(t)) = (\dot{a}(t),\dot{b}(t)) \; \text{ with } \; h(a(t),b(t)) \equiv \text{const.}
\label{eq_c_deformations}
\ee
is one-to-one and onto. The first part of the claim is obvious. For the second part consider the functions
$$
f_{b_j}(a,b) := \int_{b_j} d\ln\mu = \int_{b_j} \frac{b(\lambda)}{\nu}\frac{d\lambda}{\lambda} = \ln\mu(\alpha_j) = \pi i n_j \in \pi i \ZZ
$$
along $(a(t),b(t))$, where the $b_j$ are the $b$-cycles of $Y$. Taking the derivative yields 
$$
\frac{d}{dt} f_{b_j}(a,b)\big|_{t=0} = \int_{b_j} \frac{d}{dt}\left(\frac{b(\lambda)}{\nu}\right)\bigg|_{t=0} \frac{d\lambda}{\lambda} = \int_{b_j} \del_t (\del_{\lambda}\ln \mu) |_{t=0} \, d\lambda = 0= 0.
$$
Morover, for the $a$-cycles $a_j$ we have 
$$
f_{a_j}(a,b) = f_j(a,b) = \int_{a_j} d\ln\mu = \int_{a_j} \frac{b(\lambda)}{\nu}\frac{d\lambda}{\lambda} = 0
$$
and consequently
$$
\frac{d}{dt} f_{a_j}(a,b)\big|_{t=0} = \int_{a_j} \frac{d}{dt}\left(\frac{b(\lambda)}{\nu}\right)\bigg|_{t=0} \frac{d\lambda}{\lambda} = \int_{a_j} \del_t (\del_{\lambda}\ln \mu) |_{t=0} \, d\lambda = 0.
$$
Since all integrals of $\del_t (\del_{\lambda}\ln \mu) |_{t=0}$ vanish, there exists a meromorphic function $\phi$ with 
$$
d\phi = \del_t (\del_{\lambda}\ln \mu) |_{t=0} \,d\lambda.
$$
Due to the Whitham equation \eqref{compatibility_deformations} this function is given by $\phi = (\del_t\ln \mu)|_{t=0}= \frac{c}{\nu}$. Thus the map in \eqref{eq_c_deformations} is bijective. This shows $\dim (\ker dh) = g$ and consequently $\dim (\text{im}\,dh) = g$ as well. Therefore $dh:\RR^{2g} \to \RR^g$ has full rank and the claim follows.
\end{proof}

\section{The phase space $(M_g^{\mathbf{p}},\Omega)$}

In the following we will define the phase space of our integrable system. We need some preparation and first recall the generalized Weierstrass representation \cite{DPW}. Set
$$
\Lambda_{-1}^{\infty}\mathfrak{sl}_2(\CC) = \{ \xi_{\lambda} \in \mathcal{O}(\CC^*,\mathfrak{sl}_2(\CC)) \left|\right. (\lambda \xi_{\lambda})_{\lambda = 0} \in \CC^* \epsilon_+\}.
$$
A \textbf{potential} is a holomorphic $1$-form $\xi_{\lambda}dz$ on $\CC$ with $\xi_{\lambda} \in \Lambda_{-1}^{\infty}\mathfrak{sl}_2(\CC)$. Given such a potential one can solve the holomorphic ODE $d\phi_{\lambda} = \phi_{\lambda}\xi_{\lambda}$ to obtain a map $\phi_{\lambda}: \CC \to \Lambda_r SL(2,\CC)$. Then Theorem \ref{theorem_iwasawa} yields an extended frame $F_{\lambda}: \CC \to \Lambda_r SU(2)$ via the $r$-Iwasawa decomposition
$$
\phi_{\lambda} = F_{\lambda} B_{\lambda}.
$$
It is proven in \cite{DPW} that each extended frame can be obtained from a potential $\xi_{\lambda} dz$ by the Iwasawa decomposition. Note, that we have the inclusions
$$
\mathcal{P}_g \subset \Lambda_{-1}^{\infty}\mathfrak{sl}_2(\CC) \subset \Lambda_r \mathfrak{sl}_2(\CC).
$$
An extended frame $F_{\lambda}: \CC \to \Lambda_r SU(2)$ is of \textbf{finite type}, if there exists $g \in \NN$ such that the corresponding potential $\xi_{\lambda}dz$ satisfies $\xi_{\lambda} \in \mathcal{P}_g \subset \Lambda_{-1}^{\infty}\mathfrak{sl}_2(\CC)$. We say that a polynomial Killing field has minimal degree if and only if it has neither roots nor poles in $\lambda \in \CC^*$. We will need the following proposition that summarizes two results by Burstall-Pedit \cite{Burstall_Pedit_1, Burstall_Pedit_2}.

\begin{proposition}[\cite{Hauswirth_Kilian_Schmidt_1}, Proposition 4.5]\label{proposition_minimal_degree}
For an extended frame of finite type there exists a unique polynomial Killing field of minimal degree. There is a smooth 1:1 correspondence between the set of extended frames of finite type and the set of polynomial Killing fields without zeros.
\end{proposition}

Consider the map $A: \xi_{\lambda} \mapsto A(\xi_{\lambda}) := -\lambda \det \xi_{\lambda}$ (see \cite{Hauswirth_Kilian_Schmidt_1}) and set $\mathcal{P}_g^1(\mathbf{p}) := A^{-1}[\mathcal{M}_g^1(\mathbf{p})]$. Moreover, denote by $C_{\mathbf{p}}^{\infty} := C^{\infty}(\RR/\mathbf{p}\ZZ)$ the Frechet space of real infinitely differentiable functions of period $\mathbf{p} \in \RR^+$.  The above discussion yields an injective map 
$$
\phi: \mathcal{P}_g^1(\mathbf{p}) \subset \Lambda_{-1}^{\infty}\mathfrak{sl}_2(\CC) \to \phi[\mathcal{P}_g^1(\mathbf{p})] \subset C_{\mathbf{p}}^{\infty} \times C_{\mathbf{p}}^{\infty},\; \xi_{\lambda} \mapsto (u(\xi_{\lambda}), u_y(\xi_{\lambda})).
$$

\begin{definition}
Let $M_g^{\mathbf{p}}$ denote the space of $(u,u_y) \in C_{\mathbf{p}}^{\infty} \times C_{\mathbf{p}}^{\infty}$ (with fixed period $\mathbf{p}$) such that $(u,u_y)$ is of finite type in the sense of Def. \ref{killing1}, where $\Phi_{\lambda}$ is of fixed degree $g \in \NN_0$, and $\zeta_{\lambda}(0) \in \mathcal{P}_g^1(\mathbf{p})$ with $\zeta_{\lambda} = \Phi_{\lambda} - \lambda^{g-1}\overline{\Phi_{1/\bar{\lambda}}}^t$, i.e. $M_g^{\mathbf{p}} := \phi[\mathcal{P}_g^1(\mathbf{p})]$.
\end{definition}

Now we are able to prove the following lemma.

\begin{lemma}\label{lemma_phi_embedding}
The map $\phi: \mathcal{P}_g^1(\mathbf{p}) \to M_g^{\mathbf{p}},\; \xi_{\lambda} \mapsto (u(\xi_{\lambda}), u_y(\xi_{\lambda}))$ is an embedding.
\end{lemma}

\begin{proof}
From the previous discussion we know that $\phi: \mathcal{P}_g^1(\mathbf{p}) \to M_g^{\mathbf{p}}$ is bijective. We show that $\phi^{-1}: M_g^{\mathbf{p}} \to \mathcal{P}_g^1(\mathbf{p})$ is continuous. Assume that $g$ is the minimal degree for $\xi_{\lambda} \in \mathcal{P}^1_g(\mathbf{p})$ (see Proposition \ref{proposition_minimal_degree}). Then the Jacobi fields 
$$
(\omega_0,\del_y \omega_0),\ldots,(\omega_{g-1},\del_y\omega_{g-1}) \in C^{\infty}(\CC/\mathbf{p}\ZZ) \times C^{\infty}(\CC/\mathbf{p}\ZZ)
$$ 
are linearly independent over $\CC$ with all their derivatives up to order $2g+1$. We will now show that they stay linearly independent if we restrict them to $\RR$. For this, suppose that they are linearly dependent on $\RR$ with all their derivatives up to order $2g+1$. Since $u$ solves the elliptic $\sinh$-Gordon equation with analytic coefficients $u$ is analytic on $\CC$ \cite{Lewy}. Thus the $(\omega_i,\del_y \omega_i)$ are analytic as well since they only depend on $u$ and its $k$-th derivatives with $k \leq 2i +1 \leq 2g+1$ (see \cite{Pinkall_Sterling}, Proposition 3.1). Thus they stay linearly dependent on an open neighborhood and the subset $M \subset \CC$ of points such that these functions are linearly dependent is open and closed. Therefore $M = \CC$, a contradiction! \\

By considering all derivatives of $(u,u_y)$ up to order $2g+1$ we get a small open neighborhood $U$ of $(u,u_y) \in C^{\infty}(\RR/\mathbf{p}\ZZ) \times C^{\infty}(\RR/\mathbf{p}\ZZ)$ such that the functions
$$
(\widetilde{\omega}_0,\del_y \widetilde{\omega}_0),\ldots,(\widetilde{\omega}_{g-1},\del_y\widetilde{\omega}_{g-1}) \in C^{\infty}(\RR/\mathbf{p}\ZZ) \times C^{\infty}(\RR/\mathbf{p}\ZZ)
$$
remain linearly independent for $(\widetilde{u},\widetilde{u}_y) \in U$. Given $(u,u_y) \in U$ there exist numbers $a_0,\ldots,a_{g-1}$ such that the $g$ vectors
$$
((\omega_0(a_j),\del_y \omega_0(a_j)),\ldots,(\omega_{g-1}(a_j),\del_y\omega_{g-1}(a_j)))^t 
$$
are linearly independent. Recall that $(\omega_g,\del_y \omega_g) = \sum_{i=0}^{g-1} c_i (\omega_i,\del_y \omega_i)$ in the finite type situation, which assures the existence of a polynomial Killing field. Inserting these $a_0,\ldots,a_{g-1}$ into the equation $(\omega_g,\del_y \omega_g) = \sum_{i=0}^{g-1} c_i (\omega_i,\del_y \omega_i)$ we obtain an invertible $g\times g$ matrix and can calculate the $c_i$. This shows that the coefficients $c_i$ continuously depend on $(u,u_y) \in U$. Thus for $(u,u_y) \in M_g^{\mathbf{p}}$ and $\varepsilon > 0$ there exists a $\delta_{\varepsilon} > 0$ such that
$$
\|\xi_{\lambda}(u,u_y) - \xi_{\lambda}(\widetilde{u},\widetilde{u}_y)\| < \varepsilon
$$
holds for all Cauchy data $(\widetilde{u},\widetilde{u}_y) \in M_g^{\mathbf{p}}$ with $\|(u,u_y) - (\widetilde{u},\widetilde{u}_y)\| < \delta_{\varepsilon}$, where the norm is given by the supremum of the first $2g+1$ derivatives.
\end{proof}

Let us study the map
$$
Y: M_g^{\mathbf{p}} \to \Sigma_g^{\mathbf{p}} \simeq \mathcal{M}^1_g(\mathbf{p}), \;\; (u,u_y) \mapsto Y(u,u_y)
$$
that appears in the diagram
\[
\xymatrix{
\mathcal{P}_g^1(\mathbf{p}) \ar[d]_{\phi} \ar[dr]^A
& \\
M_g^{\mathbf{p}} \ar[r]^{Y} & \Sigma_g^{\mathbf{p}}}
\]
% Here the map $A: \mathcal{P}_g^1(\mathbf{p}) \to \Sigma_g^{\mathbf{p}} \simeq \mathcal{M}^1_g(\mathbf{p})$ is given by $\xi_{\lambda} \mapsto A(\xi_{\lambda}) := -\lambda \det \xi_{\lambda}$ (see \cite{Hauswirth_Kilian_Schmidt_1}).

\begin{proposition}\label{proposition_fiber_bundle}
The map $Y: M_g^{\mathbf{p}} \to \Sigma_g^{\mathbf{p}} \simeq \mathcal{M}^1_g(\mathbf{p}), \; (u,u_y) \mapsto Y(u,u_y)$ is a principle bundle with fiber $\text{Iso}(Y(u,u_y)) \simeq \text{Pic}_{g+1}^{\RR}(Y(u,u_y)) \simeq (\SS^1)^g$. In particular $M_g^{\mathbf{p}}$ is a manifold of dimension $2g$.
\end{proposition}

\begin{proof}
Due to Theorem \ref{theorem_Mg} the space $\mathcal{M}^1_g(\mathbf{p})$ is a smooth $g$-dimensional manifold. From Proposition 4.12 in \cite{Hauswirth_Kilian_Schmidt_1} we know that the mapping 
$$
A: \mathcal{P}_g^1(\mathbf{p}) \to \mathcal{M}_g^1(\mathbf{p}),\; \xi_{\lambda} \mapsto -\lambda \det(\xi_{\lambda})
$$
is a principal fiber bundle with fiber $(\SS^1)^g$ and thus $\mathcal{P}_g^1(\mathbf{p})$ is a manifold of dimension $2g$. Due to Lemma \ref{lemma_phi_embedding} the map $\phi: \mathcal{P}_g^1(\mathbf{p}) \to M_g^{\mathbf{p}}$ is an embedding and thus $M_g^{\mathbf{p}}$ is a manifold of dimension $2g$ as well.
\end{proof}

Note, that the structure of such ``finite-gap manifolds'' is also investigated in \cite{Dorfmeister_Wu_Loop_groups} and \cite{Bobenko_Kuksin_Sine, Kuksin_Hamiltonian_PDE}. 

\subsection{Hamiltonian formalism}
It will turn out that $M_g^{\mathbf{p}}$ can be considered as a symplectic manifold with a certain symplectic form $\Omega$. To see this, we closely follow the exposition of \cite{McKean} and consider the phase space of $(q,p) \in C_{\mathbf{p}}^{\infty} \times C_{\mathbf{p}}^{\infty}$ equipped with the symplectic form
$$
\Omega((\delta q, \delta p),(\widetilde{\delta q},\widetilde{\delta p})) = \int_0^{\mathbf{p}} \left(\delta q(x)\widetilde{\delta p}(x) - \widetilde{\delta q}(x)\delta p(x) \right)\,dx
$$
and the Poisson bracket
$$
\{f,g\} = \int_0^{\mathbf{p}} \langle\nabla f, J \nabla g\rangle \,dx \; \text{ with } \; J=\begin{pmatrix} 0 & 1 \\ -1 & 0 \end{pmatrix}.
$$
Here $f$ and $g$ are functionals of the form $h:C_{\mathbf{p}}^{\infty} \times C_{\mathbf{p}}^{\infty} \to \RR,\; (q,p)\mapsto h(q,p)$ and $\nabla h$ denotes the corresponding gradient of $h$ in the function space $C_{\mathbf{p}}^{\infty} \times C_{\mathbf{p}}^{\infty}$. Note, that there holds $\{f,g\} = \Omega(\nabla f, \nabla g)$. If we consider functionals $H,f$ on the function space $M = C_{\mathbf{p}}^{\infty} \times C_{\mathbf{p}}^{\infty}$ we have
$$
df(X) = \int_0^{\mathbf{p}} \langle \nabla f, X\rangle \, dx
$$
and
$$
X(H)= J \nabla H.
$$
Since $X(H)$ is a vector field it defines a flow $\Phi: O \subset M \times \RR \to M$ such that $\Phi((q_0,p_0),t)$ solves
$$
\frac{d}{dt}\Phi((q_0,p_0),t) = X(H)(\Phi((q_0,p_0),t)) \; \text{ with } \; \Phi((q_0,p_0),0) = (q_0,p_0).
$$
In the following we will write $(q(t),p(t))^t := \Phi((q_0,p_0),t)$ for integral curves of $X(H)$ that start at $(q_0,p_0)$. A direct calculation shows
$$
\frac{d}{dt} f(q(t),p(t))\bigg|_{t=0} = df(X(H)) = \int_0^{\mathbf{p}} \langle \nabla f, J \nabla H \rangle \, dx = \{f,H\}
$$
and we see again that $f$ is an integral of motion if and only if $f$ and $H$ are in involution. Set $(q,p) = (u,u_y)$, where $u$ is a solution of the $\sinh$-Gordon equation, i.e.
$$
\Delta u + 2\sinh(2u) = u_{xx} + u_{yy} + 2\sinh(2u) = 0.
$$
Setting $t = y$ we can investigate the so-called $\sinh$-Gordon flow that is expressed by
\[
\frac{d}{d y}\begin{pmatrix} u \\ u_y \end{pmatrix} = 
\begin{pmatrix} u_y \\ -u_{xx}-2\sinh(2u) \end{pmatrix} = 
J\nabla H_2 = \begin{pmatrix} 0 & 1 \\ -1 & 0 \end{pmatrix} \begin{pmatrix} \frac{\del H_2}{\del q} \\ \frac{\del H_2}{\del p} \end{pmatrix}
\]
with the Hamiltonian
$$
H_2(q,p) = \int_0^\mathbf{p} \tfrac{1}{2} p^2-\tfrac{1}{2}(q_x)^2+\cosh(2q) \, dx = \int_0^\mathbf{p} \tfrac{1}{2} (u_y)^2 - \tfrac{1}{2}(u_x)^2 + \cosh(2u) \, dx
$$
and corresponding gradient
$$
\nabla H_2 = \left(q_{xx}+2\sinh(2q),p\right)^t = \left(u_{xx}+2\sinh(2u),u_y\right)^t.
$$

\begin{remark}\label{remark_periodic_flow}
Since we have a loop group splitting (the $r$-Iwasawa decomposition) in the finite type situation, all corresponding flows can be integrated. Thus the flow $(q(y),p(y))^t=(u(x,y),u_y(x,y))^t$ that corresponds to the $\sinh$-Gordon flow is defined for all $y \in \RR$.
\end{remark}

Due to Remark \ref{remark_periodic_flow} $q(y) = u(x,y)$ is a periodic solution of the $\sinh$-Gordon equation with $u(x+\mathbf{p},y) = u(x,y)$ for all $(x,y) \in \RR^2$. The Hamiltonian $H_2$ is an integral of motion, another one is associated with the flow of translation (here we set $t = x$) induced by the functional
$$
H_1(q,p) = \int_0^{\mathbf{p}} pq_x \,dx = \int_0^{\mathbf{p}} u_y u_x \,dx \;\;\text{ with }\;\; \frac{d}{d x}\begin{pmatrix} u \\ u_y \end{pmatrix} = 
\begin{pmatrix} u_x \\ u_{yx}\end{pmatrix} = 
J\nabla H_1.
$$

\section{Polynomial Killing fields and integrals of motion}

Following \cite{Kilian_Schmidt_infinitesimal}, we will now describe how the functions $\varphi((\lambda,\mu),z) := F_{\lambda}^{-1}(z)v(\lambda,\mu)$ and $\psi((\lambda,\mu),z) := \left(\begin{smallmatrix} 0 & i \\ -i & 0 \end{smallmatrix}\right)\sigma^*\varphi((\lambda,\mu),z)$ can be used to describe the functions $\omega_n,\sigma_n,\tau_n$ from the Pinkall-Sterling iteration.

\begin{proposition}[\cite{Kilian_Schmidt_infinitesimal}, Proposition 3.1]\label{proposition_omega_lambda}
Define $\omega := \psi_1\varphi_1 - \psi_2\varphi_2$. Then
\begin{enumerate}
\item[(i)] The functon $h = \psi^t \varphi$ satisfies $dh=0$.
\item[(ii)] The function $\omega$ is in the kernel of the Jacobi operator and can be supplemented to a parametric Jacobi field with corresponding (up to complex constants)
$$
\tau = \frac{i\psi_2\varphi_1}{e^u},\;\;\; \sigma = -\frac{i\psi_1\varphi_2}{e^u}.
$$ 
\end{enumerate}
\end{proposition}

\begin{proposition}[\cite{Kilian_Schmidt_infinitesimal}, Proposition 3.3] \label{proposition_expansion_omega}
Let $h = \psi^t\varphi$. Then the entries of
$$
P(z) := \frac{\psi \varphi^t}{\psi^t \varphi} = \frac{1}{h}
\begin{pmatrix}
\psi_1\varphi_1 & \psi_1 \varphi_2 \\
\psi_2\varphi_1 & \psi_2 \varphi_2
\end{pmatrix}
$$
have the asymptotic expansions at $\lambda=0$
\begin{gather*}
-\frac{i\omega}{2h} = \frac{1}{\sqrt{\lambda}} \sum_{n=1}^{\infty} \omega_n (-\lambda)^n,\;\; -\frac{i\tau}{h} = \frac{1}{\sqrt{\lambda}}\sum_{n=0}^{\infty} \tau_n (-\lambda)^n, \;\;
-\frac{i\sigma}{h} = \frac{1}{\sqrt{\lambda}}\sum_{n=1}^{\infty} \sigma_n (-\lambda)^n.
\end{gather*}
\end{proposition}

%By applying the involution $\rho$ we can compute the asymptotic expansion of $P(z)$ at $\lambda =\infty$ from the expansion at $\lambda=0$. Since we have $\bar{\tau} =\sigma$ we get $\rho^*(\frac{\omega}{h}) = \frac{\omega}{h}$ and $\rho^*(\frac{\bar{\tau}}{h}) = \frac{\sigma}{h}$. This yields
%
%\begin{corollary}[\cite{Kilian_Schmidt_infinitesimal}, Corollary 3.4]
%The entries of $P(z)$ have at $\lambda=\infty$ the asymptotic expansions
%\begin{eqnarray*}
%\frac{i\omega}{2h} & = & \sqrt{\lambda} \sum_{n=1}^{\infty} (-1)^n\bar{\omega}_n \lambda^{-n}, \\
%\frac{i\tau}{h} & = & \sqrt{\lambda}\sum_{n=1}^{\infty} (-1)^n \bar{\sigma}_n \lambda^{-n}, \\
%\frac{i\sigma}{h} & = & \sqrt{\lambda}\sum_{n=0}^{\infty} (-1)^n \bar{\tau}_n \lambda^{-n}.
%\end{eqnarray*}
%\end{corollary}

\begin{definition}
Consider the asymptotic expansion 
$$
\ln\mu = \frac{1}{\sqrt{\lambda}}\frac{i\mathbf{p}}{2} + \sqrt{\lambda}\sum_{n \geq 0} c_n \lambda^{n} \;\; \text{ at } \; \lambda = 0
$$
and set $H_{2n+1} := (-1)^{n+1}\Re(c_{n})$ and $H_{2n+2} := (-1)^{n+1}\Im(c_{n})$ for $n\geq 0$.
\end{definition}

\begin{remark}
Since
$$
\ln\mu = \tfrac{1}{\sqrt{\lambda}}\tfrac{i\mathbf{p}}{2} + \sqrt{\lambda}\int_0^{\mathbf{p}}\left(-i(\del u)^2 + \tfrac{i}{2}\cosh(2u)\right)\,dt + O(\lambda)
$$
at $\lambda = 0$ we see that the functions $H_1, H_2$ are given by
\begin{eqnarray*}
H_1 & = & \int_0^{\mathbf{p}} \tfrac{1}{2} u_y u_x \,dx, \\
H_2 & = & -\int_0^\mathbf{p} \tfrac{1}{4} (u_y)^2 - \tfrac{1}{4}(u_x)^2 + \tfrac{1}{2}\cosh(2u) \, dx.
\end{eqnarray*}
These functions are proportional to the Hamiltonians that induce the flow of translation and the $\sinh$-Gordon flow respectively.
\end{remark}

We will now illustrate the link between the Pinkall-Sterling iteration from Proposition \ref{prop_pinkall_sterling} and these functions $H_n$ (which we call \textbf{Hamiltonians} from now on) and show that the functions $H_n$ are pairwise in involution. Recall the formula
$$
\frac{d}{dt}H((u,u_y) + t(\delta u,\delta u_y))|_{t=0} = dH_{(u,u_y)}(\delta u,\delta u_y) = \Omega(\nabla H(u,u_y),(\delta u,\delta u_y))
$$
from the last section. First we need the following lemma.

\begin{lemma}\label{lemma_variation_ln_mu}
For the map $\ln\mu$ we have the variational formula
$$
\frac{d}{dt}\ln\mu((u,u_y) + t(\delta u,\delta u_y))\bigg|_{t=0} = \int_0^{\mathbf{p}}\frac{1}{\varphi^t\psi}\psi^t\delta U_{\lambda} \varphi \,dx
$$
with
$$
\delta U_{\lambda} = \frac{1}{2}
\begin{pmatrix}
-i\delta u_y & i\lambda^{-1} e^u \delta u -i e^{-u}\delta u \\
i\lambda  e^u\delta u -i e^{-u}\delta u & i\delta u_y
\end{pmatrix}.
$$
\end{lemma}

\begin{proof}
We follow the ansatz presented in \cite{McKean}, Section $6$, and obtain for $F_{\lambda}(x)$ solving $\frac{d}{dx}F_{\lambda} = F_{\lambda}U_{\lambda}$ with $F_{\lambda}(0) = \unity$ the variational equation
$$
\frac{d}{dx}\frac{d}{dt}F_{\lambda}(\delta u,\delta u_y)\bigg|_{t=0} = \left(\frac{d}{dt}F_{\lambda}(\delta u,\delta u_y)\bigg|_{t=0}\right) U_{\lambda} + F_{\lambda}\delta U_{\lambda} 
$$
with 
$$
\left(\frac{d}{dt}F_{\lambda}(\delta u,\delta u_y)\bigg|_{t=0}\right)(0) = \begin{pmatrix} 0 & 0 \\ 0 & 0 \end{pmatrix}
$$
and
$$
\delta U_{\lambda} = \frac{1}{2}
\begin{pmatrix}
-i\delta u_y & i\lambda^{-1} e^u \delta u -i e^{-u}\delta u \\
i\lambda  e^u\delta u -i e^{-u}\delta u & i\delta u_y
\end{pmatrix}.
$$
The solution of this differential equation is given by
$$
\left(\frac{d}{dt}F_{\lambda}(\delta u,\delta u_y)\bigg|_{t=0}\right)(x) = \left(\int_0^x F_{\lambda}(y) \delta U_{\lambda}(y) F_{\lambda}^{-1}(y) \,dy\right) F_{\lambda}(x)
$$
and evaluating at $x = \mathbf{p}$ yields
$
\frac{d}{dt}M_{\lambda}(\delta u,\delta u_y)|_{t=0} = \left(\int_0^{\mathbf{p}} F_{\lambda}(y) \delta U_{\lambda}(y) F_{\lambda}^{-1}(y) \,dy\right) M_{\lambda}.
$
Due to Lemma \ref{lemma_var_general} there holds
$$
\delta M_{\lambda} = \frac{d}{dt}M_{\lambda}(\delta u,\delta u_y)\bigg|_{t=0}= \left[\sum_{i=1}^2 \frac{(\delta v_i) w_i^t}{w_i^t v_i},M_{\lambda}\right] + (P(\delta\mu) +\sigma^*P(\delta\mu)).
$$
If we multiply the last equation with $\frac{w_1^t}{w_1^t v_1}$ from the left and $v_1$ from the right we get
$$
\mu \frac{w_1^t\delta v_1}{w_1^t v_1} - \mu \frac{w_1^t \delta v_1}{w_1^t v_1}  + \delta \mu \frac{w_1^t v_1}{w_1^t v_1} = \delta \mu = \mu \int_0^{\mathbf{p}}\frac{1}{\varphi^t\psi}\psi^t\delta U_{\lambda} \varphi \,dx
$$
and therefore
$
\frac{d}{dt}\ln\mu(\delta u,\delta u_y)|_{t=0} = \int_0^{\mathbf{p}}\frac{1}{\varphi^t\psi}\psi^t\delta U_{\lambda} \varphi \,dx.
$
This proves the claim.
\end{proof}

We now will apply Lemma \ref{lemma_variation_ln_mu} and Corollary \ref{expansion_mu} to establish a link between solutions $\omega_n$ of the homogeneous Jacobi equation from Proposition \ref{prop_pinkall_sterling} and the Hamiltonians $H_n$.

\begin{theorem}\label{theorem_hamiltonians}
For the series of Hamiltonians $(H_n)_{n\in \NN}$ and solutions $(\omega_n)_{n\in \NN_0}$ of the homogeneous Jacobi equation from the Pinkall-Sterling iteration there holds
$$
\nabla H_{2n+1} = (\Re(\omega_n(\cdot,0)),\Re(\del_y \omega_n(\cdot,0))) \; \text{ and } \; \nabla H_{2n+2} = (\Im(\omega_n(\cdot,0)),\Im(\del_y \omega_n(\cdot,0))).
$$
\end{theorem}

\begin{proof}
Considering the result of Lemma \ref{lemma_variation_ln_mu} a direct calculation gives
\begin{eqnarray*}
\frac{d}{dt}\ln\mu(\delta u,\delta u_y)\bigg|_{t=0} & = & \int_0^{\mathbf{p}}\frac{1}{\varphi^t\psi}\psi^t\delta U_{\lambda} \varphi \,dx \\
& \stackrel{h=\varphi^t \psi}{=} & \int_0^{\mathbf{p}}\frac{1}{2h}\psi^t
\begin{pmatrix}
-i\delta u_y & i\lambda^{-1} e^u \delta u -i e^{-u}\delta u \\
i\lambda  e^u\delta u -i e^{-u}\delta u & i\delta u_y
\end{pmatrix}
\varphi \,dx \\
& = & \int_0^{\mathbf{p}} \frac{1}{2h} ((\psi_2\varphi_1(i\lambda e^u - ie^{-u}) + \psi_1\varphi_2(i\lambda^{-1}e^u-ie^{-u})) \delta u \\
& & \hspace{1.3cm} -\, i(\psi_1\varphi_1 - \psi_2\varphi_2)\delta u_y) \,dx \\
& \stackrel{\text{Prop. }\ref{proposition_omega_lambda}}{=} & \int_0^{\mathbf{p}} \frac{1}{2h} \left((\lambda e^{2u}\tau - \tau -\lambda^{-1}e^{2u}\sigma +\sigma)\delta u -i\omega \, \delta u_y \right) \,dx \\ 
& \stackrel{\text{Prop. }\ref{prop_pinkall_sterling}}{=}& \int_0^{\mathbf{p}} \frac{1}{2h} \left(-(\del\omega -\delbar\omega)\delta u -i\omega \, \delta u_y \right) \,dx \\
& = & \int_0^{\mathbf{p}} \frac{1}{2h} \left(i\omega_y \, \delta u -i\omega \, \delta u_y \right) \,dx = 
\Omega(-\tfrac{i}{2h}  (\omega, \omega_y), (\delta u, \delta u_y)) \\
& = & \Omega\left(\tfrac{-i}{2h}(\Re(\omega), \Re(\omega_y)), (\delta u, \delta u_y)\right) + i\Omega\left(\tfrac{-i}{2h}(\Im(\omega),\Im(\omega_y)), (\delta u, \delta u_y)\right) \\
& = & \sqrt{\lambda}\sum_{n\geq 0}(-1)^{n+1} \Omega\left((\Re(\omega_n), \Re(\del_y\omega_n)), (\delta u, \delta u_y)\right) \lambda^{n} \\
& & +\, i\sqrt{\lambda}\sum_{n\geq 0} (-1)^{n+1} \Omega\left((\Im(\omega_n),\Im(\del_y\omega_n)), (\delta u, \delta u_y)\right) \lambda^{n}
\end{eqnarray*}
around $\lambda = 0$ due to Proposition \ref{proposition_expansion_omega}. On the other hand we know from Corollary \ref{expansion_mu} that we have the following asymptotic expansion of $\ln\mu$ around $\lambda = 0$ 
\begin{eqnarray*}
\ln\mu & = & \tfrac{1}{\sqrt{\lambda}}\tfrac{i\mathbf{p}}{2} + \sqrt{\lambda}\int_0^{\mathbf{p}}\left(-i(\del u)^2 + \tfrac{i}{2}\cosh(2u)\right)\,dt + O(\lambda) \\
& = & \tfrac{1}{\sqrt{\lambda}}\tfrac{i\mathbf{p}}{2} + \sqrt{\lambda}\sum_{n\geq 0} c_n \lambda^{n} \\
& = & \tfrac{1}{\sqrt{\lambda}}\tfrac{i\mathbf{p}}{2} + \sqrt{\lambda}\sum_{n\geq 0} (-1)^{n+1} H_{2n+1} \lambda^n + i\sqrt{\lambda}\sum_{n\geq 0} (-1)^{n+1} H_{2n+2} \lambda^{n}.
\end{eqnarray*}
Thus we get
\begin{eqnarray*}
\frac{d}{dt}\ln\mu(\delta u,\delta u_y)\bigg|_{t=0} & = & \sqrt{\lambda}\sum_{n\geq 0} (-1)^{n+1} \Omega\left(\nabla H_{2n+1}, (\delta u, \delta u_y)\right) \lambda^{n} \\
& & +\, i\sqrt{\lambda}\sum_{n\geq 0} (-1)^{n+1} \Omega\left(\nabla H_{2n+2}, (\delta u, \delta u_y)\right) \lambda^{n}
\end{eqnarray*}
and a comparison of the coefficients of the two power series yields the claim.
\end{proof}

\section{An inner product on $\Lambda_r\mathfrak{sl}_2(\CC)$} 

We already introduced a differential operator $L_{\lambda}:= \frac{d}{dx} + U_{\lambda}$ such that the $\sinh$-Gordon flow can be expressed in commutator form, i.e.
$$
\tfrac{d}{dy} L_{\lambda} = \tfrac{d}{dy}U_{\lambda} = [L_{\lambda},V_{\lambda}] = \tfrac{d}{dx}V_{\lambda} + [U_{\lambda},V_{\lambda}].
$$

In the following we will translate the symplectic form $\Omega$ with respect to the identification $(u,u_y) \simeq U_{\lambda}$. First recall that the span of $\{\epsilon_+,\epsilon_-, \epsilon\}$ is $\mathfrak{sl}_2(\CC)$ and that the inner product 
$$
\langle \cdot,\cdot \rangle: \mathfrak{sl}_2(\CC) \times \mathfrak{sl}_2(\CC) \to \CC, \; (\alpha,\beta) \mapsto \langle \alpha,\beta \rangle := \trace(\alpha\cdot \beta)
$$
is non-degenerate. We will now extend the inner product $\langle\cdot,\cdot\rangle$ to a non-degenerate inner product $\langle\cdot,\cdot \rangle_{\Lambda}$ on $\Lambda_r\mathfrak{sl}_2(\CC) = \Lambda_r \mathfrak{su}_2(\CC) \oplus \Lambda^+_{r}\mathfrak{sl}_2(\CC)$.

% such that
%$$
%\langle\cdot,\cdot \rangle_{\Lambda}|_{\Lambda_r \mathfrak{su}_2(\CC) \times \Lambda_r \mathfrak{su}_2(\CC)} \equiv 0 \; \text{ and } \; \langle\cdot,\cdot \rangle_{\Lambda}|_{\Lambda^+_{r}\mathfrak{sl}_2(\CC) \times \Lambda^+_{r}\mathfrak{sl}_2(\CC)} \equiv 0,
%$$
%i.e. $\Lambda_r \mathfrak{su}_2(\CC)$ and $\Lambda^+_{r}\mathfrak{sl}_2(\CC)$ are isotropic subspaces of $\Lambda_r\mathfrak{sl}_2(\CC)$ with respect to $\langle\cdot,\cdot \rangle_{\Lambda}$. 

\begin{lemma}\label{lemma_manin}
The map $\langle\cdot,\cdot \rangle_{\Lambda}: \Lambda_r\mathfrak{sl}_2(\CC) \times \Lambda_r\mathfrak{sl}_2(\CC) \to \RR$ given by
$$
(\alpha,\beta) \mapsto \langle\alpha,\beta \rangle_{\Lambda} := \Im \left(\text{Res}_{\lambda = 0} \tfrac{d\lambda}{\lambda}\trace(\alpha\cdot\beta)\right)
$$
is bilinear and non-degenerate. Moreover, there holds
$$
\langle\cdot,\cdot \rangle_{\Lambda}|_{\Lambda_r \mathfrak{su}_2(\CC) \times \Lambda_r \mathfrak{su}_2(\CC)} \equiv 0 \; \text{ and } \; \langle\cdot,\cdot \rangle_{\Lambda}|_{\Lambda^+_{r}\mathfrak{sl}_2(\CC) \times \Lambda^+_{r}\mathfrak{sl}_2(\CC)} \equiv 0,
$$
i.e. $\Lambda_r \mathfrak{su}_2(\CC)$ and $\Lambda^+_{r}\mathfrak{sl}_2(\CC)$ are isotropic subspaces of $\Lambda_r\mathfrak{sl}_2(\CC)$ with respect to $\langle\cdot,\cdot \rangle_{\Lambda}$.
\end{lemma}

\begin{proof} The bilinearity of $\langle\cdot,\cdot \rangle_{\Lambda}$ follows from the bilinearity of $\trace(\cdot)$. Now consider a non-zero element $\xi = \sum_{i \in \mathcal{I}} \lambda^i \xi_i \in \Lambda_r \mathfrak{sl}_2(\CC)$ and pick out an index $j \in \mathcal{I}$ such that $\xi_j \neq 0$. Setting $\widetilde{\xi} = i \lambda^{-j} \overline{\xi}^t_j$ we obtain 
$$
\langle \xi, \widetilde{\xi}\rangle_{\Lambda} = \Im \left(\text{Res}_{\lambda = 0} \tfrac{d\lambda}{\lambda}\trace(\xi\cdot\widetilde{\xi})\right) = \trace(\xi_j \cdot \overline{\xi}^t_j) \in \RR^+
$$
since $\xi_j \neq 0$. This shows that $\langle\cdot,\cdot \rangle_{\Lambda}$ is non-degenerate, i.e. the first part of the lemma. \\

We will now prove the second part of the lemma, namely that $\Lambda_r \mathfrak{su}_2(\CC)$ and $\Lambda^+_{r}\mathfrak{sl}_2(\CC)$ are isotropic subspaces of $\Lambda_r\mathfrak{sl}_2(\CC)$ with respect to $\langle\cdot,\cdot \rangle_{\Lambda}$. \\

First we consider $\alpha^+ = \alpha^+_{\lambda} = \sum_i \lambda^i \alpha_i^+,\; \widetilde{\alpha}^+ = \sum_i \lambda^i \widetilde{\alpha}_i^+ \in \Lambda_r \mathfrak{su}_2(\CC)$ with 
$$
\alpha^+_{\lambda} = -\overline{\alpha^+_{1/\bar{\lambda}}}^t \; \text{ and } \; \widetilde{\alpha}^+_{\lambda} = -\overline{\widetilde{\alpha}^+_{1/\bar{\lambda}}}^t.
$$
Then one obtains
\begin{eqnarray*}
\langle \alpha^+,\widetilde{\alpha}^+\rangle_{\Lambda} & = & \Im \left(\text{Res}_{\lambda = 0} \tfrac{d\lambda}{\lambda}\trace(\alpha^+\cdot\widetilde{\alpha}^+)\right) \\
& = & \Im(\trace(\alpha_{-1}^+\cdot\widetilde{\alpha}_1^+ + \alpha_0^+\cdot \widetilde{\alpha}_0^+ + \alpha_1^+\cdot \widetilde{\alpha}_{-1}^+)).
\end{eqnarray*}
A direct calculation gives
\begin{eqnarray*}
\overline{\trace(\alpha_{-1}^+ \widetilde{\alpha}_1^+ + \alpha_0^+ \widetilde{\alpha}_0^+ + \alpha_1^+ \widetilde{\alpha}_{-1}^+)} & = & \trace((-\overline{\alpha}_{-1}^+)^t (-\overline{\widetilde{\alpha}}_1^+)^t + (-\overline{\alpha}_0^+)^t (-\overline{\widetilde{\alpha}}_0^+)^t + (-\overline{\alpha}_1^+)^t (-\overline{\widetilde{\alpha}}_{-1}^+)^t) \\
& \stackrel{!}{=} & \trace(\alpha_{-1}^+ \widetilde{\alpha}_1^+ + \alpha_0^+ \widetilde{\alpha}_0^+ + \alpha_1^+ \widetilde{\alpha}_{-1}^+)
\end{eqnarray*}
and thus $\trace(\alpha_{-1}^+ \widetilde{\alpha}_1^+ + \alpha_0^+ \widetilde{\alpha}_0^+ + \alpha_1^+ \widetilde{\alpha}_{-1}^+) \in \RR$. This shows
$$
\langle \alpha^+,\widetilde{\alpha}^+\rangle_{\Lambda} = \Im(\trace(\alpha_{-1}^+\cdot\widetilde{\alpha}_1^+ + \alpha_0^+\cdot \widetilde{\alpha}_0^+ + \alpha_1^+\cdot \widetilde{\alpha}_{-1}^+)) = 0.
$$

Now consider $\beta^- = \sum_{i\geq 0} \lambda^i\beta_i^-,\; \widetilde{\beta}^- = \sum_{i\geq 0} \lambda^i\widetilde{\beta}_i^- \in \Lambda^+_{r}\mathfrak{sl}_2(\CC)$ where $\beta_0^-, \widetilde{\beta}_0^-$ are of the form
$$
\beta_0^- = \begin{pmatrix} h_0 & e_0 \\ 0 & -h_0 \end{pmatrix}, \;\; \widetilde{\beta}_0^- = \begin{pmatrix} \widetilde{h}_0 & \widetilde{e}_0 \\ 0 & -\widetilde{h}_0 \end{pmatrix}
$$
with $h_0,\widetilde{h}_0 \in \RR$ and $e_0,\widetilde{e}_0 \in \CC$. Then one gets
\begin{eqnarray*}
\langle \beta^-,\widetilde{\beta}^-\rangle_{\Lambda} & = & \Im \left(\text{Res}_{\lambda = 0} \tfrac{d\lambda}{\lambda}\trace(\beta^-\cdot\widetilde{\beta}^-)\right) \\
& = & \Im(\trace(\beta_0^-\cdot \widetilde{\beta}_0^-)) \\
& = & \Im(2 h_0\widetilde{h}_0) \\
% & = & \Im \trace \begin{pmatrix} h_0 \widetilde{h}_0 & * \\ 0 & h_0 \widetilde{h}_0 \end{pmatrix} \\
& = & 0.
\end{eqnarray*}
This yields the second claim and concludes the proof.
\end{proof}

\begin{remark}
Recall the notion of \textit{Manin triples} $(\mathfrak{g},\mathfrak{p},\mathfrak{q})$ that were introduced by Vladimir Drinfeld \cite{Drinfeld}. They consist of a Lie algebra $\mathfrak{g}$ with a non-degenerate invariant symmetric bilinear form, together with two isotropic subalgebras $\mathfrak{p}$ and $\mathfrak{q}$ such that $\mathfrak{g} = \mathfrak{p} \oplus \mathfrak{q}$ as a vector space. Thus Lemma \ref{lemma_manin} shows that $(\Lambda_r\mathfrak{sl}_2(\CC), \Lambda_r \mathfrak{su}_2(\CC), \Lambda^+_{r}\mathfrak{sl}_2(\CC))$ is a Manin triple.
\end{remark}

\section{The symplectic form $\Omega$ and Serre duality}

This section incorporates our previous results and establishes a connection between the symplectic form $\Omega$ and Serre Duality \cite[Thm. 17.9]{Fo}. Moreover, we will show that $(M_g^{\mathbf{p}},\Omega,H_2)$ is a completely integrable Hamiltonian system.

\begin{definition}
Let $H^0_{\RR}(Y,\Omega) := \{\omega \in H^0(Y,\Omega) \left|\right. \overline{\eta^*\omega} = -\omega\}$ be the real part of $H^0(Y,\Omega)$ with respect to the involution $\eta$.
\end{definition}

Since $\eta$ is given by $(\lambda,\mu) \mapsto (1/\bar{\lambda},\bar{\mu})$ we have
$$
\eta^*\overline{\delta\ln\mu} = \delta\ln\mu \;\; \text{ and } \;\; \eta^*\overline{\tfrac{d\lambda}{\lambda}} = -\tfrac{d\lambda}{\lambda}.
$$
Let us define the map $\omega: T_{(u,u_y)}M_g^{\mathbf{p}} \to H^0_{\RR}(Y(u,u_y),\Omega)$ by
$$
(\delta u, \delta u_y) \mapsto \omega(\delta u, \delta u_y) := \delta \ln\mu(\delta u, \delta u_y) \frac{d\lambda}{\lambda}.
$$

\begin{remark}\label{remark_dY}
Due to Theorem \ref{theorem_Mg} we can identify the space $T_{Y(u,u_y)}\Sigma_g^{\mathbf{p}}$ of infinitesimal non-isospectral (but iso-periodic) deformations of $Y(u,u_y)$ with the space $H^0_{\RR}(Y(u,u_y),\Omega)$ via the map $c \mapsto \omega(c) := \frac{c}{\nu}\frac{d\lambda}{\lambda} = \delta\ln\mu \frac{d\lambda}{\lambda}$. Therefore $\omega$ can be identified with $dY$, the derivative of $Y: M_g^{\mathbf{p}} \to \Sigma_g^{\mathbf{p}}$. Due to Proposition \ref{proposition_fiber_bundle}, the map $Y: M_g^{\mathbf{p}} \to \Sigma_g^{\mathbf{p}}$ is a submersion. Thus the map $\omega: T_{(u,u_y)}M_g^{\mathbf{p}} \to H^0_{\RR}(Y(u,u_y),\Omega)$ is surjective.
\end{remark}

%Note that deformations which keep the period $\mathbf{p}$ fixed indeed correspond to a holomorphic 1-form $\omega$ since in that case we have
%$$
%\omega = \delta\ln\mu \frac{d\lambda}{\lambda} = \frac{c(\lambda)}{\lambda}\frac{d\lambda}{\nu} = \frac{\sum_{i=1}^g c_i\lambda^i}{\lambda}\frac{d\lambda}{\nu} = \sum_{i=1}^g c_i \frac{\lambda^{i-1} d\lambda}{\nu}.
%$$

\begin{definition}
 Let $L_{(u,u_y)} \subset T_{(u,u_y)} M_g^{\mathbf{p}}$ be the kernel of the the map $\omega: T_{(u,u_y)} M_g^{\mathbf{p}} \to H^0_{\RR}(Y(u,u_y), \Omega)$, i.e. $L_{(u,u_y)} := \ker(\omega)$.
\end{definition}

Now we are able to formulate and prove the main result. The proof is based on the ideas and methods presented in the proof of \cite{Schmidt_infinite}, Theorem 7.5.

\begin{theorem}\label{pairing}\text{ }
\begin{enumerate}
\item[(i)] There exists an isomorphism of vector spaces $d\Gamma_{(u,u_y)}: H^1_{\RR}(Y(u,u_y), \mathcal{O}) \to L_{(u,u_y)}$.
\item[(ii)] For all $[f] \in H^1_{\RR}(Y(u,u_y), \mathcal{O})$ and all $(\delta u, \delta u_y) \in T_{(u,u_y)}M_g^{\mathbf{p}}$ the equation
\begin{equation}
\Omega(d\Gamma_{(u,u_y)}([f]),(\delta u,\delta u_y)) = i\, \text{Res}([f] \, \omega(\delta u,\delta u_y))
\label{master}
\end{equation}
holds. Here the right hand side is defined as in the Serre Duality Theorem \cite{Fo}.
\item[(iii)] $(M_g^{\mathbf{p}},\Omega,H_2)$ is a Hamiltonian system. In particular, $\Omega$ is non-degenerate on $M_g^{\mathbf{p}}$.
\end{enumerate}
\end{theorem}

\begin{remark}\text{ }
\begin{enumerate}
\item[(i)]
From the Serre Duality Theorem \cite[Thm. 17.9]{Fo} we know that the pairing $\text{Res}: H^1(Y(u,u_y), \mathcal{O}) \times H^0(Y(u,u_y), \Omega) \to \CC$ is non-degenerate.
\item[(ii)]
$L_{(u,u_y)} \subset T_{(u,u_y)}M_g^{\mathbf{p}}$ is a maximal isotropic subspace with respect to the symplectic form $\Omega: T_{(u,u_y)}M_g^{\mathbf{p}} \times T_{(u,u_y)}M_g^{\mathbf{p}} \to \RR$, i.e. $L_{(u,u_y)}$ is Lagrangian.
\end{enumerate}
\end{remark}

\begin{proof}\text{}
\begin{enumerate}
\item[(i)]
Let $[f] = [(f_0,\eta^*\bar{f}_0)] \in H^1_{\RR}(Y,\mathcal{O})$ be a cocycle with representative $f_0$ as defined in Lemma \ref{Lemma_reality_H1}. Then, defining $A_{f_0}(x) := P_x(f_0) + \sigma^*P_x(f_0)$, we get
\begin{equation*}
A_{f_0}(x) = \sum_{i=0}^{g-1} c_i \lambda^{-i} (P_x(\lambda^{-1}\nu) + \sigma^*P_x(\lambda^{-1}\nu)) = \sum_{i=0}^{g-1} c_i \lambda^{-i} \zeta_{\lambda}(x)
\end{equation*}
and also
$$
\delta U_{\lambda}(x) = [A_{f_0}^+(x),L_{\lambda}(x)] = [L_{\lambda}(x),A_{f_0}^-(x)]
$$
due to Theorem \ref{theorem_vf_isospectral_U}. Moreover,
$$
\delta v(x) = -A_{f_0}^-(x) v(x)
$$
holds due to Remark \ref{remark_Af_minus} with $A_{f_0}^-(x) = -\sum \frac{(\delta v_i(x)) w_i^t(x)}{w_i^t(x) v_i(x)}$. From Lemma \ref{general_variation_U} we know that in general
$$
\delta U_{\lambda}(x) = \left[L_{\lambda}(x),-\sum_{i=1}^2 \frac{(\delta v_i(x)) w_i^t(x)}{w_i^t(x) v_i(x)}\right] + (P_x(\tfrac{\delta\ln\mu}{\mathbf{p}}) +\sigma^*P_x(\tfrac{\delta\ln\mu}{\mathbf{p}})).
$$
Since $\delta U_{\lambda}(x) = [L_{\lambda}(x),A_{f_0}^-(x)]$ we see that $\delta\ln\mu(\delta u^{f_0}, \delta u_y^{f_0}) = 0$ and consequently 
$$
(\delta u^{f_0}, \delta u_y^{f_0}) \in \ker(\omega).
$$
Thus we have an injective map $d\Gamma_{(u,u_y)}: H^1_{\RR}(Y(u,u_y), \mathcal{O}) \to L_{(u,u_y)}$. Due to Remark \ref{remark_dY} we know that $\omega: T_{(u,u_y)} M_g^{\mathbf{p}} \to H^0_{\RR}(Y(u,u_y), \Omega)$ is surjective. Since $\dim H^0_{\RR}(Y(u,u_y), \Omega) = g$ there holds $\dim L_{(u,u_y)} = \dim \ker(\omega) = g$ and thus $d\Gamma_{(u,u_y)}: H^1_{\RR}(Y(u,u_y),\mathcal{O}) \to L_{(u,u_y)}$ is an isomorphism of vector spaces.

\item[(ii)] 
For an isospectral variation of $U_{\lambda}(x)$ we have
\begin{equation*}
\delta U_{\lambda}(x) = [L_{\lambda},B^-(x)] = \tfrac{d}{dx}B^-(x) - [U_{\lambda}(x),B^-(x)]
\end{equation*}
with a map $B^-(x): \RR \to \Lambda_r^+ \mathfrak{sl}_2(\CC)$. In the following $\oplus$ will denote the $\Lambda_r \mathfrak{su}_2(\CC)$-part of $\Lambda_r \mathfrak{sl}_2(\CC) = \Lambda_r \mathfrak{su}_2(\CC) \oplus \Lambda_r^+ \mathfrak{sl}_2(\CC)$ and $\ominus$ will correspond to the second summand $\Lambda_r^+ \mathfrak{sl}_2(\CC)$. Since $\delta U_{\lambda}(x)$ lies in the $\oplus$-part and $\frac{d}{dx}B^-(x)$ 
% as well as $\widehat{P}_x(\tfrac{\delta\ln\mu}{\mathbf{p}}) := P_x(\tfrac{\delta\ln\mu}{\mathbf{p}}) + \sigma^*P_x(\tfrac{\delta\ln\mu}{\mathbf{p}})$ 
lies in the $\ominus$-part, we see that $\delta U_{\lambda}(x)$ is equal to the $\oplus$-part of the commutator expression $-[U_{\lambda}(x),B^-(x)]$. Writing $U$ for $U_{\lambda}(x)$ and $B^-$ for $B^-(x)$ we get
\begin{eqnarray*}
\delta U & = & \lambda^{-1}\delta U_{-1} + \lambda^0\delta U_0 + \lambda\delta U_1 \\
& \stackrel{!}{=} & \big(\lambda^{-1}[B_0^-,U_{-1}] + \lambda^0([B_1^-,U_{-1}] + [B_0^-,U_0]) \\
& & \hspace{1mm} +\, \lambda([B_2^-,U_{-1}]+[B_1^-,U_0]+[B_0^-,U_1])\big)_{\oplus}.
\end{eqnarray*}
Thus we arrive at three equations
\begin{gather*}
\delta U_{-1} = [B_0^-,U_{-1}], \;\;\; \delta U_0 = \left([B_1^-,U_{-1}] + [B_0^-,U_0]\right)_{\oplus}, \\
\delta U_1 = \left([B_2^-,U_{-1}]+[B_1^-,U_0]+[B_0^-,U_1]\right)_{\oplus}.
\end{gather*}
Recall that $U_{\lambda}$ is given by
$$
U_{\lambda} = \frac{1}{2}
\begin{pmatrix}
 -i u_y & i\lambda^{-1}e^u + ie^{-u} \\
i\lambda e^u + ie^{-u} & iu_y
\end{pmatrix}
$$
and consequently
$$
\delta U_{\lambda} = \frac{1}{2}
\begin{pmatrix}
-i\delta u_y & i\lambda^{-1} e^u \delta u -i e^{-u}\delta u \\
i\lambda  e^u\delta u -i e^{-u}\delta u & i\delta u_y
\end{pmatrix}.
$$
We can now use the above equations in order to obtain relations on the coefficients $B_0^-$ and $B_1^-$ of $B^- = \sum_{i\geq 0} \lambda^i B_i^-$ where $B_0^-$ is of the form
$$
B_0^- = \begin{pmatrix} h_0 & e_0 \\ 0 & -h_0 \end{pmatrix} \; \text{ with } \; h_0 \in \RR,\, e_0 \in \CC,
$$
and $B_1^-$ is of the form $B_1^- = \left(\begin{smallmatrix} h_1 & e_1 \\ f_1 & -h_1 \end{smallmatrix}\right)$. Since $\delta U_{-1} = [B_0^-,U_{-1}]$ a direct calculation yields
\begin{gather*}
B_0^-=\begin{pmatrix}\tfrac{1}{2}\delta u & e_0 \\ 0 & -\tfrac{1}{2}\delta u \end{pmatrix}.
\end{gather*}
Moreover, the sum $[B_1^-,U_{-1}] + [B_0^-,U_0]$ is given by
$$
\begin{pmatrix}
\tfrac{i}{2}e_0 e^{-u} - \tfrac{i}{2}f_1 e^u & ih_1 e^u + ie_0 u_y + \tfrac{i}{2}e^{-u}\delta u \\
-\tfrac{i}{2}e^{-u}\delta u & -\tfrac{i}{2}e_0 e^{-u} + \tfrac{i}{2}f_1 e^u
\end{pmatrix}.
$$
For the diagonal entry of $[B_1^-,U_{-1}] + [B_0^-,U_0]$ the $\oplus$-part is given by the imaginary part and therefore
$$
-\tfrac{i}{2}\delta u_y = \tfrac{i}{2}\Re(e_0) e^{-u} - \tfrac{i}{2}\Re(f_1) e^u.
$$
Thus we get
$$
\delta u_y = \Re(f_1) e^u - \Re(e_0) e^{-u}.
$$
For $B^-(x) := A_{f_0}^-(x)$ we obtain
\begin{eqnarray*}
\Im \left(\trace(\delta U_0 A_{f_0, 0}^{-}) + \trace(\delta U_{-1}A_{f_0 ,1}^{-})\right) & = & \Im\left(-i\delta u_y h_0 - \tfrac{i}{2}e_0 e^{-u}\delta u + \tfrac{i}{2}f_1 e^u\delta u\right) \\
& = & -\delta u_y \Re(h_0) +\tfrac{1}{2} \delta u \left(\Re(f_1) e^u- \Re(e_0) e^{-u} \right) \\
& = & \dfrac{1}{2}\left( \delta u \delta u_y^{f_0} - \delta u^{f_0} \,\delta u_y\right).
\end{eqnarray*}
Now a direct calculation gives
\begin{eqnarray*}
\frac{1}{2}\Omega(d\Gamma_{(u,u_y)}([f]),(\delta u,\delta u_y)) & = & \frac{1}{2}\int_0^{\mathbf{p}} (\delta u^{f_0}\delta u_y - \delta u \,\delta u_y^{f_0}) dx \\
& = & -\int_0^{\mathbf{p}} \Im \left(\trace(\delta U_0 A_{f_0, 0}^{-}) + \trace(\delta U_{-1}A_{f_0 ,1}^{-})\right) dx \\
& = & -\int_0^{\mathbf{p}} \langle \delta U_{\lambda}(x),A_{f_0}^-(x) \rangle_{\Lambda} dx  \\
& \stackrel{\delta U_{\lambda} \in \oplus}{=} & -\int_0^{\mathbf{p}} \langle \delta U_{\lambda}(x),A_{f_0}(x) \rangle_{\Lambda} dx.
\end{eqnarray*}
Setting $\widehat{P}_x(\tfrac{\delta\ln\mu}{\mathbf{p}}) := P_x(\tfrac{\delta\ln\mu}{\mathbf{p}}) + \sigma^*P_x(\tfrac{\delta\ln\mu}{\mathbf{p}})$, we further obtain 
\begin{eqnarray*}
\frac{1}{2}\Omega(d\Gamma_{(u,u_y)}([f]),(\delta u,\delta u_y)) & \stackrel{\text{Lem. } \ref{general_variation_U}}{=} & -\int_0^{\mathbf{p}} \langle [L_{\lambda}(x),B^-(x)],A_{f_0}(x) \rangle_{\Lambda} dx \\
& & -\, \int_0^{\mathbf{p}} \langle \widehat{P}_x(\tfrac{\delta\ln\mu}{\mathbf{p}}), A_{f_0}(x) \rangle_{\Lambda} dx \\
& = & \int_0^{\mathbf{p}} \langle [B^-(x),L_{\lambda}(x)], A_{f_0}(x) \rangle_{\Lambda} dx \\
& & -\, \int_0^{\mathbf{p}} \langle \widehat{P}_x(\tfrac{\delta\ln\mu}{\mathbf{p}}), A_{f_0}(x) \rangle_{\Lambda} dx.
\end{eqnarray*}
Recall, that $\trace([B^-(x), L_{\lambda}(x)]\cdot A_{f_0}(x)) = \trace(B^-(x)\cdot[L_{\lambda}(x),A_{f_0}(x)])$. Moreover, there holds $[L_{\lambda}(x), A_{f_0}(x)] = 0$ and we get
\begin{eqnarray*}
\frac{1}{2}\Omega(d\Gamma_{(u,u_y)}([f]),(\delta u,\delta u_y)) & = & \int_0^{\mathbf{p}} \langle B^-(x), [L_{\lambda}(x), A_{f_0}(x)]\rangle_{\Lambda} dx \\
& & -\, \int_0^{\mathbf{p}} \langle \widehat{P}_x(\tfrac{\delta\ln\mu}{\mathbf{p}}), A_{f_0}(x) \rangle_{\Lambda} dx \\
& = & -\, \int_0^{\mathbf{p}} \langle \widehat{P}_x(\tfrac{\delta\ln\mu}{\mathbf{p}}), A_{f_0}(x) \rangle_{\Lambda} dx \\
& = & -\int_0^{\mathbf{p}} \langle A_{f_0}(x),\widehat{P}_x(\tfrac{\delta\ln\mu}{\mathbf{p}})\rangle_{\Lambda} dx.
\end{eqnarray*}
Writing out the last equation yields
\begin{eqnarray*}
\Omega(d\Gamma_{(u,u_y)}([f]),(\delta u,\delta u_y)) & = & -\dfrac{2}{\mathbf{p}}\int_0^{\mathbf{p}} \Im\left(\text{Res}_{y_0} \tfrac{d\lambda}{\lambda}\trace(A_{f_0}(x)\widehat{P}_x(\delta\ln\mu))\right) dx \\
& = & -2\Im\left(\text{Res}_{y_0} (f_0 \cdot \delta \ln\mu \tfrac{d\lambda}{\lambda})\right) \\
& = & i\left(\text{Res}_{y_0} (f_0 \cdot \delta \ln\mu \tfrac{d\lambda}{\lambda}) - \text{Res}_{y_0} \overline{(f_0 \cdot \delta \ln\mu \tfrac{d\lambda}{\lambda})}\right)  \\
& = & i\left(\text{Res}_{y_0} (f_0 \cdot \delta \ln\mu \tfrac{d\lambda}{\lambda}) - \text{Res}_{y_{\infty}} \eta^*\overline{(f_0 \cdot \delta \ln\mu \tfrac{d\lambda}{\lambda})}\right)  \\
& = & i\left(\text{Res}_{y_0} (f_0 \cdot \delta \ln\mu \tfrac{d\lambda}{\lambda}) + \text{Res}_{y_{\infty}} (f_{\infty} \cdot \delta \ln\mu \tfrac{d\lambda}{\lambda})\right)
\end{eqnarray*}
and thus
$$
\Omega(d\Gamma_{(u,u_y)}([f]),(\delta u, \delta u_y)) = i\, \text{Res}([f]\,\omega(\delta u,\delta u_y)).
$$
\item[(iii)]
In order to prove (iii) we have to show that $\Omega$ is non-degenerate on $M_g^{\mathbf{p}}$. From equation \eqref{master} we know that 
$$
\Omega((\delta u, \delta u_y), (\delta \widetilde{u}, \delta \widetilde{u}_y)) = 0 \; \text{ for } \; (\delta u, \delta u_y), (\delta \widetilde{u}, \delta \widetilde{u}_y) \in L_{(u,u_y)} = \ker(\omega).
$$
Moreover, $\omega: T_{(u,u_y)}M_g^{\mathbf{p}} \to H^0_{\RR}(Y(u,u_y),\Omega)$ is surjective since $\dim T_{(u,u_y)}M_g^{\mathbf{p}} = 2g$ and $\dim\ker(\omega) = g = \dim H^0_{\RR}(Y(u,u_y),\Omega)$. Thus we have 
$$
T_{(u,u_y)}M_g^{\mathbf{p}}/\ker(\omega) \simeq H^0_{\RR}(Y(u,u_y),\Omega)
$$
and there exists a basis $\{ \delta a_1,\ldots,\delta a_g,\delta b_1,\ldots,\delta b_g\}$ of $T_{(u,u_y)}M_g^{\mathbf{p}}$ such that
$$
\text{span}\{\delta a_1,\ldots,\delta a_g\} = \ker(\omega) \; \text{ and } \; \omega[\text{span}\{\delta b_1,\ldots,\delta b_g\}] = H^0_{\RR}(Y(u,u_y),\Omega).
$$
Now $L_{(u,u_y)} = \ker(\omega) \simeq H^1_{\RR}(Y(u,u_y),\mathcal{O})$ and since the pairing from Serre duality is non-degenerate we obtain with equation \eqref{master} (after choosing the appropriate basis)
$$
\Omega(\delta a_i,\delta b_j) = \delta_{ij} \; \text{ and } \; \Omega(\delta b_i,\delta a_j) = -\delta_{ij}.
$$ 
Summing up the matrix representation $B_{\Omega}$ of $\Omega$ on $T_{(u,u_y)}M_g^{\mathbf{p}}$ has the form 
$$
B_{\Omega} = \begin{pmatrix} 0 & \unity \\ -\unity & (\ast) \end{pmatrix}
$$
and thus $\Omega$ is of full rank. This shows (iii) and concludes the proof of Theorem \ref{pairing}.
\end{enumerate}
\end{proof}

\bibliography{diss_bibliography} 

\begin{thebibliography}{10}

\bibitem{Audin}
M.~Audin.
\newblock {\em {Hamiltonian systems and their integrability.}}
\newblock Providence, RI: American Mathematical Society (AMS); Paris:
  Soci\'et\'e Math\'ematique de France (SMF), 2008.

\bibitem{Bobenko_Its_Matveev}
E.~D. Belokolos, A.~I. Bobenko, V.~Z. Enol'skij, A.~R. Its, and V.~B. Matveev.
\newblock {\em {Algebro-geometric approach to nonlinear integrable equations.}}
\newblock Berlin: Springer-Verlag, 1994.

\bibitem{Bobenko_Kuksin_Sine}
A.~{Bobenko} and S.~{Kuksin}.
\newblock {Small-amplitude solutions of the sine-Gordon equation on an interval
  under Dirichlet or Neumann boundary conditions.}
\newblock {\em {J. Nonlinear Sci.}}, 5(3):207--232, 1995.

\bibitem{bob3}
A.~I. Bobenko.
\newblock {All constant mean curvature tori in $\RR^3, \SS^3$ and $\HH^3$ in
  terms of theta-functions.}
\newblock {\em Math. Ann.}, 290(2):209--245, 1991.

\bibitem{bob2}
A.~I. Bobenko.
\newblock {Constant mean curvature surfaces and integrable equations.}
\newblock {\em Russ. Math. Surv.}, 46(4):1--45, 1991.

\bibitem{bob1}
A.~I. Bobenko.
\newblock {Surfaces in terms of 2 by 2 matrices. Old and new integrable cases.}
\newblock In {\em {Fordy, Allan P. (ed.) et al., Harmonic maps and integrable
  systems. Based on conference, held at Leeds, GB, May 1992. Braunschweig:
  Vieweg. Aspects Math. E23, 83-127}}, 1994.

\bibitem{bob4}
A.~I. Bobenko.
\newblock {Exploring surfaces through methods from the theory of integrable
  systems: the Bonnet problem.}
\newblock In {\em {Surveys on geometry and integrable systems. Based on the
  conference on integrable systems in differential geometry, Tokyo, Japan, July
  17--21, 2000}}, pages 1--53. Tokyo: Mathematical Society of Japan, 2008.

\bibitem{Burstall_Pedit_1}
F.~{Burstall} and F.~{Pedit}.
\newblock {Harmonic maps via Adler-Kostant-Symes theory.}
\newblock In {\em {Harmonic maps and integrable systems. Based on conference,
  held at Leeds, GB, May 1992}}, pages 221--272. Braunschweig: Vieweg, 1994.

\bibitem{Burstall_Pedit_2}
F.~{Burstall} and F.~{Pedit}.
\newblock {Dressing orbits of harmonic maps.}
\newblock {\em {Duke Math. J.}}, 80(2):353--382, 1995.

\bibitem{Carberry_Schmidt}
E.~{Carberry} and M.~U. {Schmidt}.
\newblock {The Closure of Spectral Data for Constant Mean Curvature Tori in
  $\SS^3$}.
\newblock {\em To appear in J. Reine Angew. Math.}, Feb. 2012.

\bibitem{DPW}
J.~Dorfmeister, F.~Pedit, and H.~Wu.
\newblock {Weierstrass type representation of harmonic maps into symmetric
  spaces.}
\newblock {\em {Commun. Anal. Geom.}}, 6(4):633--668, 1998.

\bibitem{Dorfmeister_Wu_Loop_groups}
J.~{Dorfmeister} and H.~{Wu}.
\newblock {Constant mean curvature surfaces and loop groups.}
\newblock {\em {J. Reine Angew. Math.}}, 440:43--76, 1993.

\bibitem{Drinfeld}
V.~{Drinfel'd}.
\newblock {Quantum groups.}
\newblock {Proc. Int. Congr. Math., Berkeley/Calif. 1986, Vol. 1, 798-820},
  1987.

\bibitem{Fo}
O.~Forster.
\newblock {\em {Lectures on Riemann surfaces.}}
\newblock {Graduate Texts in Mathematics, Vol. 81. New York - Heidelberg
  -Berlin: Springer-Verlag. VIII}, 1981.

\bibitem{Pedit_Schmitt}
A.~Gerding, F.~Pedit, and N.~Schmitt.
\newblock {Constant mean curvature surfaces: an integrable systems
  perspective.}
\newblock In {\em {Harmonic maps and differential geometry. A harmonic map fest
  in honour of John C. Wood's 60th birthday, Cagliari, Italy, September 7--10,
  2009}}, pages 7--39. Providence, RI: American Mathematical Society (AMS),
  2011.

\bibitem{Grinevich_Schmidt}
P.~G. Grinevich and M.~U. Schmidt.
\newblock {Period preserving nonisospectral flows and the moduli space of
  periodic solutions of soliton equations.}
\newblock {\em Physica D}, 87(1-4):73--98, 1995.

\bibitem{Haak_Schmidt_Schrader}
G.~Haak, M.~Schmidt, and R.~Schrader.
\newblock {Group theoretic formulation of the Segal-Wilson approach to
  integrable systems with applications.}
\newblock {\em {Rev. Math. Phys.}}, 4(3):451--499, 1992.

\bibitem{Hauswirth_Kilian_Schmidt_1}
L.~{Hauswirth}, M.~{Kilian}, and M.~U. {Schmidt}.
\newblock {Finite type minimal annuli in $\mathbb{S}^2 \times \mathbb{R}$}.
\newblock {\em To appear in Illinois J. Math.}, Oct. 2012.

\bibitem{Hauswirth_Kilian_Schmidt_2}
L.~{Hauswirth}, M.~{Kilian}, and M.~U. {Schmidt}.
\newblock {Properly embedded minimal annuli in $\mathbb{S}^2 \times
  \mathbb{R}$}.
\newblock {\em ArXiv e-print 1210.5953}, Oct. 2012.

\bibitem{Hauswirth_Kilian_Schmidt_3}
L.~{Hauswirth}, M.~{Kilian}, and M.~U. {Schmidt}.
\newblock {The geometry of embedded constant mean curvature tori in the
  3-sphere via integrable systems}.
\newblock {\em ArXiv e-print 1309.4278}, Sept. 2013.

\bibitem{Hitchin_Harmonic}
N.~J. Hitchin.
\newblock {Harmonic maps from a 2-torus to the 3-sphere.}
\newblock {\em J. Differ. Geom.}, 31(3):627--710, 1990.

\bibitem{Kilian_Schmidt_cylinders}
M.~{Kilian} and M.~U. {Schmidt}.
\newblock {On the moduli of constant mean curvature cylinders of finite type in
  the $3$-sphere}.
\newblock {\em ArXiv e-print 0712.0108v2}, Dec. 2007.

\bibitem{Kilian_Schmidt_infinitesimal}
M.~Kilian and M.~U. Schmidt.
\newblock {On infitesimal deformations of CMC surfaces of finite type in the
  3-sphere.}
\newblock In {\em {Riemann surfaces, harmonic maps and visualization.
  Proceedings of the 16th Osaka City University International Academic
  Symposium, Osaka, Japan, December 15--20, 2008}}, pages 231--248. Osaka:
  Osaka Municipal Universities Press, 2010.

\bibitem{Knopf_phd}
M.~Knopf.
\newblock {\em Periodic solutions of the sinh-Gordon equation and integrable
  systems}.
\newblock PhD thesis, 2013.

\bibitem{Krichever_algebraic}
I.~M. Krichever.
\newblock {Methods of algebraic geometry in the theory of nonlinear equations.}
\newblock {\em {Russ. Math. Surv.}}, 32(6):185--213, 1977.

\bibitem{Kuksin_Hamiltonian_PDE}
S.~{Kuksin}.
\newblock {\em {Analysis of Hamiltonian PDE's.}}
\newblock Oxford: Oxford University Press, 2000.

\bibitem{Lewy}
H.~{Lewy}.
\newblock {Neuer Beweis des analytischen Charakters der L\"osungen elliptischer
  Differentialgleichungen.}
\newblock {\em {Math. Ann.}}, 101:609--619, 1929.

\bibitem{McIntosh_Toda}
I.~McIntosh.
\newblock {Global solutions of the elliptic 2D periodic Toda lattice.}
\newblock {\em {Nonlinearity}}, 7(1):85--108, 1994.

\bibitem{McIntosh_harmonic_tori}
I.~McIntosh.
\newblock {Harmonic tori and their spectral data.}
\newblock In {\em {Surveys on geometry and integrable systems. Based on the
  conference on integrable systems in differential geometry, Tokyo, Japan, July
  17--21, 2000}}, pages 285--314. Tokyo: Mathematical Society of Japan, 2008.

\bibitem{McKean}
H.~P. McKean.
\newblock {The sine-Gordon and sinh-Gordon equations on the circle.}
\newblock {\em {Commun. Pure Appl. Math.}}, 34:197--257, 1981.

\bibitem{Pinkall_Sterling}
U.~Pinkall and I.~Sterling.
\newblock {On the classification of constant mean curvature tori.}
\newblock {\em {Ann. Math. (2)}}, 130(2):407--451, 1989.

\bibitem{Pressley_Segal}
A.~Pressley and G.~Segal.
\newblock {\em {Loop groups.}}
\newblock Oxford (UK): Clarendon Press, repr. with corrections edition, 1988.

\bibitem{Schmidt_infinite}
M.~U. Schmidt.
\newblock {\em {Integrable systems and Riemann surfaces of infinite genus.}}
\newblock Memoirs of the American Mathematical Society Series, Number 581,
  1996.

\end{thebibliography}
\bibliographystyle{abbrv}

\end{document}